\documentclass[reqno, 12pt]{amsart}

\pdfoutput=1
\makeatletter
\let\origsection=\section \def\section{\@ifstar{\origsection*}{\mysection}}
\def\mysection{\@startsection{section}{1}\z@{.7\linespacing\@plus\linespacing}{.5\linespacing}{\normalfont\scshape\centering\S}}
\makeatother

\usepackage{mathrsfs}
\usepackage{mathabx}\changenotsign
\usepackage{dsfont}
\usepackage{mathrsfs}
\usepackage{amsmath, amssymb, amsthm, amsxtra, amsfonts, mathtools}

\usepackage{xcolor}
\usepackage[backref]{hyperref}
\hypersetup{
	colorlinks,
	linkcolor={red!60!black},
	citecolor={green!60!black},
	urlcolor={blue!60!black}
}

\usepackage[open,openlevel=2,atend]{bookmark}

\usepackage[abbrev,msc-links,backrefs]{amsrefs}
\usepackage{doi}

\renewcommand{\PrintDOI}[1]{\doi{#1}}

\usepackage[T1]{fontenc}
\usepackage{lmodern}
\usepackage[babel]{microtype}
\usepackage[english]{babel}

\usepackage{geometry}
\numberwithin{equation}{section}
\numberwithin{figure}{section}

\usepackage{enumitem}

\let\polishlcross=\l
\def\l{\ifmmode\ell\else\polishlcross\fi}

\let\emptyset=\varnothing
\let\setminus=\smallsetminus

\makeatletter
\def\moverlay{\mathpalette\mov@rlay}
\def\mov@rlay#1#2{\leavevmode\vtop{   \baselineskip\z@skip \lineskiplimit-\maxdimen
		\ialign{\hfil$\m@th#1##$\hfil\cr#2\crcr}}}
\newcommand{\charfusion}[3][\mathord]{
	#1{\ifx#1\mathop\vphantom{#2}\fi
		\mathpalette\mov@rlay{#2\cr#3}
	}
	\ifx#1\mathop\expandafter\displaylimits\fi}
\makeatother

\DeclareFontFamily{U}  {MnSymbolC}{}
\DeclareSymbolFont{MnSyC}         {U}  {MnSymbolC}{m}{n}
\DeclareFontShape{U}{MnSymbolC}{m}{n}{
	<-6>  MnSymbolC5
	<6-7>  MnSymbolC6
	<7-8>  MnSymbolC7
	<8-9>  MnSymbolC8
	<9-10> MnSymbolC9
	<10-12> MnSymbolC10
	<12->   MnSymbolC12}{}
\DeclareMathSymbol{\powerset}{\mathord}{MnSyC}{180}

\usepackage{tikz}
\usetikzlibrary{calc,decorations.pathmorphing}
\pgfdeclarelayer{background}
\pgfdeclarelayer{foreground}
\pgfdeclarelayer{front}
\pgfsetlayers{background,main,foreground,front}

\usepackage{subcaption}
\let\epsilon=\varepsilon

\let\rho=\varrho
\let\theta=\vartheta

\newcommand{\cP}{\mathcal{P}}

\newtheoremstyle{note}  {4pt}  {4pt}  {\sl}  {}  {\bfseries}  {.}  {.5em}          {}
\newtheoremstyle{introthms}  {3pt}  {3pt}  {\itshape}  {}  {\bfseries}  {.}  {.5em}          {\thmnote{#3}}
\newtheoremstyle{remark}  {2pt}  {2pt}  {\rm}  {}  {\bfseries}  {.}  {.3em}          {}

\theoremstyle{plain}
\newtheorem{theorem}{Theorem}[section]
\newtheorem{lemma}[theorem]{Lemma}
\newtheorem{example}[theorem]{Example}
\newtheorem{prop}[theorem]{Proposition}

\newtheorem{claim}[theorem]{Claim}

\theoremstyle{note}
\newtheorem{dfn}[theorem]{Definition}

\theoremstyle{remark}

\usepackage{accents}

\usepackage{lineno}
\newcommand*\patchAmsMathEnvironmentForLineno[1]{
	\expandafter\let\csname old#1\expandafter\endcsname\csname #1\endcsname
	\expandafter\let\csname oldend#1\expandafter\endcsname\csname end#1\endcsname
	\renewenvironment{#1}
	{\linenomath\csname old#1\endcsname}
	{\csname oldend#1\endcsname\endlinenomath}}
\newcommand*\patchBothAmsMathEnvironmentsForLineno[1]{
	\patchAmsMathEnvironmentForLineno{#1}
	\patchAmsMathEnvironmentForLineno{#1*}}
\AtBeginDocument{
	\patchBothAmsMathEnvironmentsForLineno{equation}
	\patchBothAmsMathEnvironmentsForLineno{align}
	\patchBothAmsMathEnvironmentsForLineno{flalign}
	\patchBothAmsMathEnvironmentsForLineno{alignat}
	\patchBothAmsMathEnvironmentsForLineno{gather}
	\patchBothAmsMathEnvironmentsForLineno{multline}
}

\newtheoremstyle{case}{}{}{}{}{}{:}{ }{}
\theoremstyle{case}

\newcommand{\qedge}[7]{
	
	\ifx\relax#4\relax
	\def\qoffs{0pt}
	\else
	\def\qoffs{#4}
	\fi
	
	\def\qhedge{
		($#1+#3!\qoffs!-90:#2-#3$) --
		($#2+#1!\qoffs!-90:#3-#1$) --
		($#3+#2!\qoffs!-90:#1-#2$) -- cycle}

	\coordinate (12) at ($#1!\qoffs!90:#2$);
	\coordinate (13) at ($#1!\qoffs!-90:#3$);
	\coordinate (23) at ($#2!\qoffs!90:#3$);
	\coordinate (21) at ($#2!\qoffs!-90:#1$);
	\coordinate (31) at ($#3!\qoffs!90:#1$);
	\coordinate (32) at ($#3!\qoffs!-90:#2$);
	
	\def\nqhedge{
		(13) let \p1=($(13)-#1$), \p2=($(12)-#1$) in
		arc[start angle={atan2(\y1,\x1)}, delta angle={atan2(\y2,\x2)-atan2(\y1,\x1)-360*(atan2(\y2,\x2)-atan2(\y1,\x1)>0)}, x radius=\qoffs, y radius=\qoffs] --
		(21) let \p1=($(21)-#2$), \p2=($(23)-#2$) in
		arc[start angle={atan2(\y1,\x1)}, delta angle={atan2(\y2,\x2)-atan2(\y1,\x1)-360*(atan2(\y2,\x2)-atan2(\y1,\x1)>0)}, x radius=\qoffs, y radius=\qoffs] --
		(32) let \p1=($(32)-#3$), \p2=($(31)-#3$) in
		arc[start angle={atan2(\y1,\x1)}, delta angle={atan2(\y2,\x2)-atan2(\y1,\x1)-360*(atan2(\y2,\x2)-atan2(\y1,\x1)>0)}, x radius=\qoffs, y radius=\qoffs] --
		cycle}
	
	\ifx\relax#5\relax
	\def\qlwidth{1pt}
	\else
	\def\qlwidth{#5}
	\fi
	
	\ifx\relax#7\relax
	\fill \nqhedge;
	\else
	\fill[#7]\nqhedge;
	\fi
	
	\ifx\relax#6\relax
	\draw[line width=\qlwidth,rounded corners=\qoffs]\nqhedge;
	\else
	\draw[line width=\qlwidth,#6]\nqhedge;
	\fi
}

\newsavebox\vdegbox
\savebox\vdegbox{\tikz{
		\draw[black,fill=black] (90:1) circle (.35);
		\draw[black,line width=0.10cm] (210:1) circle (.30);
		\draw[black,line width=0.10cm] (330:1) circle (.30);
		\draw[opacity=0] (0:1.2) circle (0.1);
}}

\newsavebox\vvbox
\savebox\vvbox{\tikz{
		\draw[black,line width=0.10cm] (90:1) circle (.30);
		\draw[black,fill=black] (210:1) circle (.35);
		\draw[black,fill=black] (330:1) circle (.35);
		\draw[opacity=0] (0:1.2) circle (0.1);
}}

\newsavebox\pdegbox
\savebox\pdegbox{\tikz{
		\draw[black,line width=0.10cm] (90:1) circle (.30);
		\draw[black,fill=black] (210:1) circle (.35);
		\draw[black,fill=black] (330:1) circle (.35);
		\draw[black,line width=0.28cm ] (210:1) -- (330:1);
		\draw[opacity=0] (0:1.2) circle (0.1);
}}

\newsavebox\vvvbox
\savebox\vvvbox{\tikz{
		\draw[black,fill=black] (90:1) circle (.35);
		\draw[black,fill=black] (210:1) circle (.35);
		\draw[black,fill=black] (330:1) circle (.35);
		\draw[opacity=0] (0:1.2) circle (0.1);
}}

\newsavebox\evbox
\savebox\evbox{\tikz{
		\draw[black,fill=black] (90:1) circle (.35);
		\draw[black,fill=black] (210:1) circle (.35);
		\draw[black,fill=black] (330:1) circle (.35);
		\draw[black,line width=0.28cm ] (210:1) -- (330:1);
		\draw[opacity=0] (0:1.2) circle (0.1);
}}

\newsavebox\eebox
\savebox\eebox{\tikz{
		\draw[black,fill=black] (90:1) circle (.35);
		\draw[black,fill=black] (210:1) circle (.35);
		\draw[black,fill=black] (330:1) circle (.35);
		\draw[black,line width=0.28cm ] (90:1) -- (330:1);
		\draw[black,line width=0.28cm ] (90:1) -- (210:1);
		\draw[opacity=0] (0:1.2) circle (0.1);
}}

\newsavebox\eeebox
\savebox\eeebox{\tikz{
		\draw[black,fill=black] (90:1) circle (.35);
		\draw[black,fill=black] (210:1) circle (.35);
		\draw[black,fill=black] (330:1) circle (.35);
		\draw[black,line width=0.28cm ] (90:1) -- (330:1);
		\draw[black,line width=0.28cm ] (90:1) -- (210:1);
		\draw[black,line width=0.28cm ] (210:1) -- (330:1);
		\draw[opacity=0] (0:1.2) circle (0.1);
}}

\begin{document}

\title[A pair degree condition for Hamiltonian cycles in $3$-graphs]{A pair degree condition for Hamiltonian cycles in $3$-uniform hypergraphs}
\author[B.~Sch\"{u}lke]{Bjarne Sch\"{u}lke}
\address{Fachbereich Mathematik, Universit\"{a}t Hamburg, Hamburg, Germany}
\email{bjarne.schuelke@uni-hamburg.de}


\begin{abstract}
	We prove a new sufficient pair degree condition for tight Hamiltonian cycles in~$3$-uniform hypergraphs that (asymptotically) improves the best known pair degree condition due to R\"odl, Ruci\'nski, and Szemer\'edi. 
	For graphs, Chv\'atal characterised all those sequences  of integers for which every pointwise larger (or equal) degree sequence guarantees the existence of a Hamiltonian cycle. 
	A step towards Chv\'atal's theorem was taken by P\'osa, who improved on Dirac's tight minimum degree condition for Hamiltonian cycles by showing that a certain weaker condition on the degree sequence of a graph already yields a Hamiltonian cycle.
	
	In this work, we take a similar step towards a full characterisation of all pair degree matrices that ensure the existence of tight Hamiltonian cycles in~$3$-uniform hypergraphs by proving a~$3$-uniform analogue of P\'osa's result. In particular, our result strengthens the asymptotic version of the result by R\"odl, Ruci\'nski, and Szemer\'edi.
\end{abstract}

\maketitle

\section{Introduction}

The search for conditions ensuring the existence of Hamiltonian cycles in graphs has been one of the main themes in graph theory.
For graphs, several classic results exist, starting with the necessary condition by Dirac~\cite{Dirac} stating that every graph $G=(V,E)$ on at least~$3$ vertices and with minimum degree $\delta (G) \geq \vert V\vert /2$ contains a Hamiltonian cycle.
P\'osa~\cite{Posa} improved this result to a condition on the degree sequence: 
\begin{theorem}\label{Posa_theorem}
	Let $G=([n],E)$ be a graph on $n\geq 3$ vertices with degree sequence ${d(1) \leq \dots \leq d(n)}$. 
	If~$d(i) \geq i+1$ for all $i < (n-1)/2$ and if furthermore $d\left(\left\lceil n/2 \right\rceil\right) \geq \left\lceil n/2\right\rceil$ when~$n$ is odd, then~$G$ contains a Hamiltonian cycle. 
\end{theorem}
Finally, Chv\'atal \cite{Chvatal} achieved an even stronger result: A graph~${G=([n],E)}$ on~${n\geq 3}$ vertices with degree sequence $d(1) \leq \dots \leq d(n)$ contains a Hamiltonian cycle if for all~${i < \frac{n}{2}}$ we have: $d(i) \leq i \Rightarrow d(n-i) \geq n-i$.
On the other hand, for any sequence~${a_1\leq\dots\leq a_n<n}$ not satisfying this condition there exists a graph on vertex set~$[n]$ with~$a_i\leq d(i)$ for all~$i\in [n]$ that does not contain a Hamiltonian cycle.

One can also investigate Hamiltonian cycles in more general structures: A $k$-\textit{uniform hypergraph} (or $k$-\textit{graph}) is a pair $(V,E)$ consisting of a (vertex) set $V$ and an (edge) set~${E\subseteq V^{(k)}}$.
We sometimes write~$v(H)=\vert V(H)\vert$ and~$e(H)=\vert E(H)\vert$.
In the following let $H=(V,E)$ be a~$3$-graph.
For~$U\subseteq V$, we define~${H[U]:= \left( U,E(U)\right)}$ with~$E(U):= \{e\in E: e\subseteq U\}$.
For vertices~$v,w \in V$, we denote by~${d(v,w) := \left\vert\{x\in V: vwx \in E\}\right\vert}$ the \textit{pair degree}, where for convenience we write an edge as $vwx$ instead of $\{v,w,x\}$.
In addition, it is also common to study the \textit{vertex degree}~$d(v) := \left\vert\{ e\in E: v\in e\}\right\vert$.
The minimum pair degree is~${\delta _2 (H) := \min_{vw \in V^{2}}d(v,w)}$ and the minimum vertex degree is~$\delta_1 (H) := \min_{v\in V}d(v)$.
Often it is useful to consider something like a~$2$-uniform projection of~$H$ with respect to a vertex~$v\in V$; we define the \textit{link graph} $L_v$ of $v$ as the graph $\left(V, \left\{xy: xyv \in E\right\}\right)$.

We will follow the definition of paths and cycles in~\cite{5/9}, suggested by Katona and Kierstead in~\cite{Katona_Kierstead}.
A~$3$-graph~$P$ is a \textit{tight path} of length~$\ell$, if~$\left\vert V(P)\right\vert = \ell+2$ and there is an ordering of the vertices~$V(P) = \{x_1,\dots , x_{\ell+2}\}$ such that~$E(P)=\{x_ix_{i+1}x_{i+2}:i\in[\ell]\}$.
The tuple~$\left( x_1,x_2\right)$ is the \textit{starting pair} of~$P$, the tuple~$\left( x_{\ell+1},x_{\ell+2}\right)$ is the \textit{ending pair} of~$P$, and both are the \textit{end-pairs} of~$P$ and we say that $P$ is a tight~$\left( x_1,x_2\right)$-$\left( x_{\ell+1},x_{\ell+2}\right)$-path.
All other vertices of~$P$ are called \textit{internal}.
We sometimes identify a path with the sequence of its vertices~$x_1,\dots , x_{\ell+2}$.
Accordingly, a \textit{tight cycle}~$C$ of length~$\ell\geq 4$ consists of a path~$x_1,\dots , x_{\ell}$ of length $\ell-2$ together with the two hyperedges~$x_{\ell-1}x_{\ell}x_1$ and~$x_{\ell}x_1x_2$. 
A \textit{tight walk} of length~$\ell$ is a hypergraph~$W$ with~$V(W)=\left\{x_1,\dots , x_{\ell+2}\right\}$, where the~$x_i$ are not necessarily distinct, and~$E(W)=\left\{ x_ix_{i+1}x_{i+2}: i\in [\ell] \right\}$.
Note that the length of a path, a cycle or a walk is the number of its edges and we will use this convention for cycles, paths, and walks in graphs as well.

One might also consider degree conditions for loose Hamiltonian cycles in~$k$-uniform hypergraphs, in which consecutive edges intersect in less than~$k-1$ vertices.
Loose Hamiltonian cycles were for instance studied in~\cite{loose_vertex,loose_min_deg,loose_co_deg,loose_vertex2}.
From now on we only consider tight paths and cycles and consequently we may omit the prefix ``tight''.

In recent years, there has been some progress to achieve Dirac like results for hypergraphs.
R\"odl, Ruci\'nski, and Szemer\'edi~\cite{Dirac_type} started by showing that for~$\alpha > 0$, there is some~$n_0$ such that every~$3$-graph on~$n\geq n_0$ vertices with minimum pair degree at least~$(\frac{1}{2} + \alpha ) n$ contains a Hamiltonian cycle. 
Actually, in~\cite{Dirac_type_improved} they improved the result to the following.
\begin{theorem}\label{thm:minpairdeg}
	Let $H$ be a~$3$-graph on $n$ vertices, where~$n$ is sufficiently large.
	If $H$ satisfies~${\delta_2(H) \geq \left\lfloor n/2\right\rfloor}$, then~$H$ has a Hamiltonian cycle.
	Moreover, for every~$n$, there exists an $n$-vertex~$3$-graph~$H_n$ such that ${\delta_2\left( H_n\right) = \left\lfloor n/2\right\rfloor - 1}$ and~$H_n$ does not have a Hamiltonian cycle.
\end{theorem}
More recently, Reiher, R\"odl, Ruci\'nski, Schacht, and Szemer\'edi~\cite{5/9} proved the following asymptotically optimal result.
\begin{theorem}\label{5/9_theorem}
	For every~$\alpha > 0$, there is an~$n_0\in\mathds{N}$ such that every~$3$-graph~$H$ on~$n\geq n_0$ vertices with $\delta_1 (H) \geq \left( \frac{5}{9} + \alpha\right) \frac{n^2}{2}$ contains a Hamiltonian cycle.
\end{theorem}
Since the first version of this article, this has been generalised to all~$k$ independently by Lang and Sanhueza-Matamala~\cite{LSM} and by Polcyn, Reiher, Rödl, and myself~\cite{tyh}.

In this work, we study a new pair degree condition that forces large~$3$-graphs to contain a Hamiltonian cycle.
Call a matrix~$(d_{ij})_{ij}$ \emph{Hamiltonian} if every~$3$-graph~$H=\left([n],E\right)$ with~$d(i,j)\geq d_{ij}$, for all~$ij\in[n]^{(2)}$, contains a Hamiltonian cycle.
It would be very desirable to get a result for~$3$-graphs similar to the one by Chv\'atal for degree sequences in graphs, that is, a characterisation of all Hamiltonian matrices.
For the graph case, P\'osa's result (Theorem \ref{Posa_theorem}) was a step towards the characterisation by Chv\'atal. 
In a sense, our main result can be seen as a~$3$-uniform (asymptotic) analogue of the theorem by P\'osa.
\begin{theorem}[Main result]\label{main_theorem}
	For $\alpha > 0$, there exists an~$n_0\in\mathds{N}$ such that for all~$n\in\mathds{N}$ with~$n\geq n_0$, the following holds.
	If~${H=\left( [n],E\right)}$ is a~$3$-graph with~$d(i,j)\geq \min\left( i,j,\frac{n}{2}\right) + \alpha n$ for all~${ij \in [n]^{(2)}}$, then~$H$ contains a (tight) Hamiltonian cycle.
\end{theorem}
This result strengthens the asymptotic version of Theorem~\ref{thm:minpairdeg} achieved in~\cite{Dirac_type}.

Let us remark that recently there have also been related results on degree sequences in graphs. For example, Treglown \cite{Treglown} gave a degree sequence condition that forces the graph to contain a clique factor and Staden and Treglown \cite{StadenTreglown} proved a degree sequence condition that forces the graph to contain the square of a Hamiltonian cycle.
Since the first version of this article, Bowtell and Hyde~\cite{BoHy} obtained a degree sequence condition for perfect matchings in~$3$-graphs.

Note that in the proof (and the proofs of the lemmas) we can always assume $\alpha \ll 1$. Before we start with the outline of the proof of Theorem \ref{main_theorem} in the next section, we give the following examples showing that our result is asymptotically optimal in some regard.

\begin{example}
	\renewcommand{\labelenumi}{\theenumi}
	\renewcommand{\theenumi}{(\roman{enumi})}%
	\begin{enumerate}
		\item\label{it: onethird} Consider the partition $X \dot\cup Y = [n]$ with $X = \left[\left\lceil \frac{n+1}{3}\right\rceil\right]$ and let $H$ be the hypergraph on $[n]$ containing all triples $e\in V^{(3)}$ with $\left\vert e\cap X\right\vert \neq 2$.
		
		Then we have~$d(i,j)\geq \min\left( i,j,\frac{n}{2}\right) - 1$ for all~${ij \in [n]^{(2)}}$. 
		However, if there was a Hamiltonian cycle~$C$ in~$H$, it would contain at least one edge with two vertices from~$X$. But such an edge can only lie in a cycle in which all vertices are from~$X\subsetneq [n]$. 
		Hence,~$H$ does not contain a Hamiltonian cycle.
		\item Next, look at the partition $X \dot\cup Y = [n]$ with $X = \left[\left\lfloor \frac{n}{2}\right\rfloor\right]$ and let $H$ be the hypergraph on $[n]$ containing all triples $e\in V^{(3)}$ such that $\left\vert e\cap Y\right\vert \neq 2$.
		
		Then for all ${ij \in [n]^{(2)}}$, we have $d(i,j)\geq \frac{n}{2} - 2$. But an analogous argument as above shows that $H$ does not contain a Hamiltonian cycle.
	\end{enumerate}
\end{example}
The two examples show that Theorem \ref{main_theorem} does not hold when replacing the degree condition with $d(i,j)\geq \min\left( i,j,\frac{n}{2}\right) -1$ (not even when replacing it with~$d(i,j)\geq \min\left( i,j\right) -1$) and neither when replacing it with~$d(i,j)\geq \min\left( i,j,\frac{n}{2}-2\right)$.
Note that this means that Theorem~\ref{main_theorem} cannot (asymptotically) be improved on by decreasing the requirement on the degree of every pair and neither by ``capping'' at a lower value than at~$\frac{n}{2}-2$.
However, it is not yet a Chv\'atal like characterisation of all Hamiltonian matrices.
For instance, it is easy to see that there are Hamiltonian matrices with~$d_{ij}=0$ for some~$i,j\in[n]$.

In the following, we will omit rounding issues if they are not important, e.g., we will assume that~$\alpha n$ etc.~are natural numbers.
Further, for~$A,B\subseteq \mathds{R}_+$, we write that a statement~$\mathfrak{S}$ holds for all~$a\in A$ and~$b\in B$ with~$a\ll b$, to say that for every~$b\in B$, there exists an~$a_0\in\mathds{R}_+$ such that for all~$a\in A$ with~$a\leq a_0$, the statement~$\mathfrak{S}$ holds.

\subsection*{Organisation} In the next section we give an overview of the proof, state the auxiliary results for each step and finally deduce the main result Theorem \ref{main_theorem} from these. Sections~\ref{sec_connecting_lemma}-\ref{sec_long_path} are devoted to the proofs of the auxiliary results. In the end, we collect some interesting related problems in Section \ref{concluding_remarks}.

\section{Overview and Final Proof}

The proof of Theorem~\ref{main_theorem} uses the absorption method introduced by R\"odl, Ruci\'nski, and Szemer\'edi in \cite{Dirac_type}, which helps to reduce the problem of finding a Hamiltonian cycle to the problem of constructing a cycle containing almost all vertices.

This strategy proceeds by constructing a cycle containing almost all vertices of the hypergraph~$H$ and a special subpath into which we can ``absorb'' any small set of vertices, meaning we can integrate the left-over vertices into this subpath to obtain a Hamiltonian cycle.
For that, we use that for every vertex~$v\in V(H)$, there exist many absorbers in~$H$, a structure consisting of several paths which can be restructured into paths containing~$v$ while keeping the same end-pairs.
Then, utilising the probabilistic method, we can construct an absorbing path, a path containing many absorbers for every vertex.
Lastly, we build a long path in the remainder of~$H$, consisting of almost all vertices, and connect it with the absorbing path to a cycle into which the left-over vertices can be absorbed.

For these constructions we often need to connect two paths, that is, find a path between their end-pairs.
Hence, we will begin by showing that we can connect every pair of pairs of vertices by a large number of paths with a fixed length. 
\begin{lemma}[Connecting Lemma]\label{connecting_lemma}
	Let~$\alpha,\vartheta>0$,~$n,L\in\mathds{N}$ with~$1/n\ll\vartheta\ll1/L\ll \alpha$.
	If~${H=([n],E)}$ is a~$3$-graph with~${d(i,j)\geq \min\left( i,j,\frac{n}{2}\right) + \alpha n}$, for all~${ij \in [n]^{(2)}}$, then for all disjoint ordered pairs of distinct vertices~${(x,y), (w,z)\in [n]^2}$, the number of paths of length~$L$ in~$H$ connecting~$(x,y)$ and~$(w,z)$ is at least~$\vartheta n^{L-2}$.
\end{lemma}
See Section \ref{sec_connecting_lemma} for the proof of Lemma \ref{connecting_lemma}.

Later, we will use this result whenever we need to connect different paths that have been constructed before.
However, when we want to connect paths after almost all the vertices are covered by paths, we need to ensure that there still exist paths, disjoint to all previously built paths.
To this end, we will take a special selection of vertices - the reservoir - aside, with the property that for every pair of pairs of vertices, we still have many paths of fixed length connecting them, where all internal vertices of those paths are vertices of the reservoir.
The existence of such a set will be shown by the probabilistic method.
\begin{lemma}[Reservoir Lemma]\label{reservoir_lemma}
	Let~$\alpha,\vartheta>0$ and~$n,L\in\mathds{N}$ such that~$1/n\ll\vartheta\ll 1/L\ll\alpha$.
	If~$H=([n],E)$ is a~$3$-graph satisfying~$d(i,j)\geq \min\left( i,j,\frac{n}{2}\right) + \alpha n$, for all~${ij \in [n]^{(2)}}$, then there exists a reservoir set~${\mathcal{R} \subseteq [n]}$ with~${\frac{\vartheta ^2}{2}n \leq \left\vert \mathcal{R} \right\vert \leq \vartheta ^2 n}$ such that for all disjoint ordered pairs of distinct vertices~$(x,y),(w,z)\in[n]^{2}$, there are at least~${\vartheta \left\vert \mathcal{R} \right\vert ^{L-2}/2}$ paths of length~$L$ in~$H$ which connect~$(x,y)$ and~$(w,z)$ and whose internal vertices all belong to~$\mathcal{R}$.
\end{lemma}
It follows that removing a few vertices from the reservoir will not destroy its connectability property.
\begin{lemma}[Preservation of the Reservoir]\label{reservoir_preservation_lemma}
	Let~$\alpha,\vartheta>0$ and~$n,L\in\mathds{N}$ such that~$1/n\ll\vartheta\ll 1/L\ll\alpha$. 
	If~${H=([n],E)}$ is a~$3$-graph satisfying~$d(i,j)\geq \min\left( i,j,\frac{n}{2}\right) + \alpha n$, for all~${ij \in [n]^{(2)}}$,~$\mathcal{R}$ is given by Lemma~\ref{reservoir_lemma}, and~${\mathcal{R}' \subseteq \mathcal{R}}$ with~$\vert\mathcal{R}'\vert\leq2\vartheta ^4 n$, then for all disjoint ordered pairs of distinct vertices~$(x,y),(w,z)\in[n]^2$, there is an~$(x,y)$-$(w,z)$-path of length~$L$ in~$H$ with all internal vertices belonging to~$\mathcal{R}\setminus \mathcal{R}'$.
\end{lemma}
See Section \ref{sec_reservoir} for the proof of Lemma \ref{reservoir_lemma} and Lemma \ref{reservoir_preservation_lemma}.

The proof will continue with the definition of the absorbers and we will show that for each vertex, there are many absorbers. We make use of this fact when we show that a small random selection of tuples still contains many absorbers for every $v\in V(H)$. With the Connecting Lemma we can afterwards connect all the small paths in that selection to a path that can absorb any small set of vertices.
\begin{lemma}[Absorbing Path]\label{absorbing_path_lemma}
	Let~$\alpha,\vartheta>0$ and~$n,L\in\mathds{N}$ such that~$1/n\ll\vartheta\ll 1/L\ll\alpha$.
	If~$H=([n],E)$ is a~$3$-graph satisfying~$d(i,j)\geq \min\left( i,j,\frac{n}{2}\right) + \alpha n$, for all~${ij \in [n]^{(2)}}$, and~$\mathcal{R}$ is given by Lemma~\ref{reservoir_lemma}, then there exists a path~$P_A \subseteq H\setminus\mathcal{R}$ with~$v(P_A) \leq \vartheta n$ and with the (absorbing) property that for each~${X\subseteq [n]}$ with~${\left \vert X \right \vert \leq 2\vartheta^2 n}$, there is a path with vertex set~$X \cup V(P_A)$ and the same end-pairs as $P_A$.
\end{lemma}
See Section \ref{sec_absorbing_path} for the proof of Lemma \ref{absorbing_path_lemma}.

By using weak hypergraph regularity and then an explicit result to obtain an almost perfect matching in the reduced hypergraph, we show in Section~\ref{sec_long_path} that in every hypergraph~$H$ satisfying the degree condition in Theorem~\ref{main_theorem}, there exists a path which contains almost all vertices of~$H$ (see Proposition~\ref{long_path_lemma}).
\begin{prop}[Long Path]\label{long_path_lemma}
	Let~$\alpha,\vartheta>0$ and~$n,L\in\mathds{N}$ such that~$1/n\ll\vartheta\ll 1/L\ll\alpha$.
	Let~${H=\left([n],V\right)}$ be a~$3$-graph with~$d(i,j)\geq \min\left( i,j,\frac{n}{2}\right) + \alpha n$, for all~${ij \in [n]^{(2)}}$, let~$\mathcal{R}$ be as in Lemma~\ref{reservoir_lemma}, and~$P_A$ as in Lemma~\ref{absorbing_path_lemma}.
	
	Then there exists a path~$Q\subseteq H\setminus P_A$ such that $$v(Q) \geq \left(1-2\vartheta ^2\right) n - v\left( P_A\right)$$ and~$\left\vert V(Q) \cap \mathcal{R}\right\vert \leq \vartheta ^4 n$.
\end{prop}
See Section \ref{sec_long_path} for the proof of Proposition \ref{long_path_lemma}.

Now we are ready to prove our main result, Theorem \ref{main_theorem} (see also Figure~\ref{fig: overview}).

\begin{proof}[Proof of Theorem \ref{main_theorem}]
	Let~$\alpha,\vartheta>0$ and~$n,L\in\mathds{N}$ such that~$1/n\ll\vartheta\ll 1/L\ll\alpha$. 
	Now let~${H=\left( [n],E\right)}$ be a~$3$-graph satisfying the degree condition~$d(i,j)\geq \min\left( i,j,\frac{n}{2}\right) + \alpha n$ for all~${ij \in [n]^{(2)}}$.
	Lemmas \ref{reservoir_lemma}, \ref{absorbing_path_lemma}, and Proposition \ref{long_path_lemma} provide a reservoir~$\mathcal{R}$, an absorbing path~$P_A\subseteq H\setminus\mathcal{R}$ and a long path~$Q\subseteq H\setminus P_A$ with~$\left\vert\mathcal{R}\cap V(Q)\right\vert \leq \vartheta ^4 n$.
	Let~$(a,b)$ and~$(c,d)$ be the end-pairs of~$P_A$ and let~$(r,s)$ and~$(t,u)$ be the end-pairs of~$Q$ (note that they are disjoint since we have~$Q\subseteq H\setminus P_A$).
	Since~$\left\vert\mathcal{R}\cap V(Q)\right\vert \leq \vartheta ^4 n$ and~$P_A\subseteq H\setminus\mathcal{R}$ and by Lemma~\ref{reservoir_preservation_lemma}, we can choose a path~$P_1$ of length~$L$ connecting~$(t,u)$ and~$(a,b)$ with all internal vertices in~$\mathcal{R}\setminus \left(V(Q)\cup V\left(P_A\right)\right)$ and, by the hierarchy of constants, we also find a path~$P_2$ of length~$L$ connecting~$(c,d)$ and~$(r,s)$ with all internal vertices in~$\mathcal{R}\setminus \left( V(Q)\cup V\left(P_A\right)\cup V\left( P_1\right)\right)$.
	That leaves us with a cycle~$C$ in~$H$ which satisfies~$v(C) \geq \left( 1-2\vartheta ^2\right) n$ and~$P_A \subseteq C$.
	The absorbing property of~$P_A$ guarantees that for~$X:=[n] \setminus V(C)$, there exists a path~$P_A'$ with~$V\left( P_A'\right) = V\left( P_A\right) \cup X$ which has the same end-pairs as~$P_A$ (which are connected to~$Q$) and hence there is a Hamiltonian cycle in~$H$.
\end{proof}

\begin{figure}
	\centering
	\begin{tikzpicture}[scale=0.8]
	
	\coordinate (a) at (0,0);
	\coordinate (b) at (0,1);
	
	\coordinate (c) at (0,5);
	\coordinate (d) at (0,6);
	
	\coordinate (r) at (0,8);
	\coordinate (s) at (1,8);
	
	\coordinate (t) at (15,8);
	\coordinate (u) at (16,8);
	\coordinate (rest) at (5,5);
	\coordinate (reservoir1) at (0,7);
	\coordinate (reservoir2) at (8,4);
	
	\begin{pgfonlayer}{front}
	
	\foreach \a in {a, b, c, d, r, s, t, u}
	\fill  (\a) circle (2pt);
	\node at (0.4,0) {$a$};
	\node at (0.4,1) {$b$};
	\node at (0.4,5) {$c$};
	\node at (0.4,6) {$d$};
	
	\node at (0,8.4) {$r$};
	\node at (1,8.4) {$s$};
	\node at (15,8.4) {$t$};
	\node at (16,8.4) {$u$};
	
	\node at (-0.4,3) {$P_A$};
	\node at (8,8.8) {$Q$};
	\node at (rest) {$\mathrm{leftover}$};
	\node at (10,4) {$\subseteq \mathcal{R}$};
	\node at (0.8,7) {$\subseteq \mathcal{R}$};
	
	\node at (2,3) {$\mathrm{absorbing}$};
	
	\end{pgfonlayer}
	
	\begin{pgfonlayer}{background}
	
	\draw[black!40!white, line width=2pt] (a) -- (b);
	\draw[black!40!white, line width=2pt] (c) -- (d);
	
	\draw[black!40!white, line width=2pt] (r) -- (s);
	\draw[black!40!white, line width=2pt] (t) -- (u);
	
	\draw[black!40!white, line width=2pt, decorate, decoration={snake,segment length=30,amplitude=6}] (b) -- (c);
	\draw[black!40!white, line width=2pt, decorate, decoration={snake,segment length=30,amplitude=6}] (s) -- (t);
	
	\draw[yellow!100!white, line width=2pt, decorate, decoration={snake,segment length=20,amplitude=4}] (d) -- (r);
	\draw[yellow!100!white, line width=2pt, decorate, decoration={snake,segment length=30,amplitude=6}] (a) -- (u);
	
	\draw (rest) circle (1.2 cm);
	\draw[blue!75!black, line width=1pt] (reservoir1) ellipse (5pt and 0.9cm);
	\draw[rotate around={-63.5:(reservoir2)},blue!75!black, line width=1pt] (reservoir2) ellipse (12pt and 8.8cm);
	
	\draw (0.2,2) .. controls (2,3) .. (4.4,4.4);
	\draw (0,3.8) .. controls (2,3) .. (4.4,4.4);
	
	\end{pgfonlayer}		
	
	\end{tikzpicture}
	\caption{Overview of the proof}\label{fig: overview}
\end{figure}

\section{Connecting Lemma}\label{sec_connecting_lemma}

Before we start with the actual proof of Lemma \ref{connecting_lemma}, let us take a look at the strategy.
Say, we want to connect two (ordered) pairs $(x,y)$ and $(w,z)$ in a hypergraph~$H$ satisfying the condition in Theorem~\ref{main_theorem}.
One can easily reduce the case of both pairs being arbitrary to that of both having pair degree at least~$\frac{n}{2} + \alpha n$ by ``climbing up'' in the degree sequence (see the beginning of the proof).
Then~$N\left((x,y),(w,z)\right)$, the set of common neighbours of~$(x,y)$ and~$(w,z)$, is non-empty because of the high pair degrees of~$(x,y)$ and~$(w,z)$. 
If we were able to find many ($2$-uniform)~$y$-$w$-paths in the link graphs of elements in~$N\left((x,y),(w,z)\right)$, we could subsequently insert the elements of~$N\left((x,y),(w,z)\right)$ at every third position of such a path, thereby obtaining a~$3$-uniform walk. 

So we could indeed connect two pairs if the link graphs of vertices in~$N\left((x,y),(w,z)\right)$ would inherit the right degree condition, i.e., if the vertices would be large (regarded as elements of~$\mathds{N}$).
However, since we cannot control how large the elements in~$N\left((x,y),(w,z)\right)$ are, the degree condition that the link graphs of vertices in~${N\left((x,y),(w,z)\right)}$ inherit may not be strong enough to let us connect two vertices by ``climbing up'' the degree sequence.
The idea to insert a middle pair~$(a,b)$, as done in~\cite{5/9}, overcomes this problem.
If~$(a,b)$ has some large common neighbours with~$(x,y)$ and some with~$(w,z)$, we can find enough~$(x,y)$-$(w,z)$ walks passing through~$(a,b)$ by applying the strategy explained above (now we can connect vertices in the link graphs by ``climbing up'' the degree sequence).
The number of those walks will depend on the number of large common neighbours that~$(a,b)$ has with each~$(x,y)$ and~$(w,z)$.
So roughly speaking, if the sum over all~$(a,b)$ of large common neighbours of~$(a,b)$ and~$(x,y)$ and of~$(a,b)$ and~$(w,z)$ is large, we can indeed prove the Connecting Lemma.
This last point (in its accurate form) will follow from the observation that two link graphs of large vertices have many common edges.

Note that this strategy can be used in the seemingly different settings of our pair degree condition and the minimum vertex degree condition in~\cite{5/9}, since in both cases we have ``well connected'' subgraphs in every link graph and each two of these subgraphs intersect in many edges: In~\cite{5/9} those subgraphs are the \textit{robust subgraphs} and in our case we can just consider the link graphs of large vertices.
After the first version of this article, this idea has also been used extensively in~\cite{tyh}.
\begin{proof}[Proof of Lemma \ref{connecting_lemma}]
	Observe that when we show that there exists an~$L\in\mathds{N}$ and a~$\vartheta>0$ such that the statement of Lemma~\ref{connecting_lemma} holds for these, it easily follows that it holds for all~$L\in\mathds{N}$ and~$\vartheta>0$ with~$1/n\ll\vartheta\ll 1/L\ll\alpha\ll1$.
	Hence, let the hierarchy and~$H$ be given as described in the lemma and let~$(x,y),(w,z)\in [n]^2$ be two disjoint ordered pairs of distinct vertices.
	
	First, we will show that it is possible to ``climb up'' along the degree sequence in (compared to~$n$) few steps, starting from the pairs $(x,y)$ and $(w,z)$ and ending with pairs of vertices~$\geq \frac{n}{2}$.
	
	In the second step, we will connect these two by utilising an analogous ``climb up'' argument in the link graphs of neighbours of a pair and slipping in an additional connective pair.
	We first look for walks rather than paths and conclude by remarking that many of them will actually be paths.
	
	\subsection*{First Step}
	%
	%
	%
	%
	%
	%
	%
	%
	%
	%
	%
	%
	%
	%
	
	By induction on $\ell \geq 3$, we will prove the following statement: There exist at least~${\left( \frac{\alpha}{5} \right) ^{\ell-2} n^{\ell-2}}$ walks $x_1=x, x_2=y, x_3, \dots , x_{\ell}$ such that for $i\geq 3$ we have: 
	\begin{align}\label{climb_up_cond}
	x_i \geq \min \left( \frac{\alpha}{4}n(i-2),\frac{n}{2}\right) +\frac{\alpha}{4}n
	\end{align}
	We will first show the statement for~$\ell =3$ and~$\ell = 4$ and then deduce it for any~$\ell \geq 5$ given that it holds for~$\ell -1$.
	
	$\ell=3:$ By the degree condition on $H$ we have $d(x,y) \geq \min \left( 1,2, \frac{n}{2}\right)+\alpha n$. Hence, there exist at least $\frac{\alpha}{5}n$ possible vertices $x_3$ such that $x_1,x_2,x_3$ is a walk and $x_3 \geq \frac{\alpha}{4}n+\frac{\alpha}{4}n$.
	
	$\ell=4:$ Let $x_1,x_2,x_3$ be one of those $\frac{\alpha}{5} n$ walks satisfying the condition (\ref{climb_up_cond}) that we get by the previous case. We then have $d(x_2,x_3) \geq \min \left(1,\frac{\alpha}{2}n, \frac{n}{2}\right)+\alpha n$, so there exist at least~$\frac{\alpha}{5}n$ possible vertices~$x_4$ such that~$x_1,x_2,x_3,x_4$ is a walk and~$x_i \geq \frac{\alpha}{4}n(i-2)+\frac{\alpha}{4}n$ for~$i=3,4$.
	
	$\ell\geq 5:$ Let $x_1,x_2,x_3, \dots ,x_{\ell-1}$ be one of the $\left( \frac{\alpha}{5} \right) ^{\ell-3} n^{\ell-3}$ walks satisfying, for~$i\geq 3$, $$x_i \geq \min \left( \frac{\alpha}{4}n(i-2),\frac{n}{2}\right) +\frac{\alpha}{4}n$$ that we get by induction.
	Then our pair degree condition entails $$d(x_{\ell-2},x_{\ell-1}) \geq \min \left( \frac{\alpha}{4}n(\ell-4)+\frac{\alpha}{4}n, \frac{n}{2}\right)+\alpha n$$ which in turn gives rise to at least $\frac{\alpha}{5}n$ possible vertices $x_{\ell}$ such that $x_1,x_2,\dots, x_{\ell}$ build a walk and we have~${x_i \geq \min\left(\frac{\alpha}{4}n(i-2),\frac{n}{2}\right)+\frac{\alpha}{4}n}$ for all~${i \in [\ell], i \geq 3}$.
	
	This leaves us with~$\left( \frac{\alpha}{5}\right) ^{ \frac{2}{\alpha}} n^{\frac{2}{\alpha}}$ possibilities for walks $$x_1=x, x_2=y, x_3, \dots , x_{\frac{2}{\alpha} +2}$$ with $x_{\frac{2}{\alpha}+1},x_{\frac{2}{\alpha} +2} \geq \frac{n}{2}$ and an analogous argument for $(w,z)$ with just as many possibilities for walks $$z_1=z, z_2=w, z_3, \dots , z_{ \frac{2}{\alpha}+2}$$ with $z_{\frac{2}{\alpha} +1},z_{ \frac{2}{\alpha} +2} \geq \frac{n}{2}$.

	\subsection*{Second Step}
	\begin{figure}
		\centering
		\begin{tikzpicture}[scale=1.05]
		
		\coordinate (x') at (0,0);
		\coordinate (y') at (0.9,0);
		
		\coordinate (r1) at (1.8,0.8);
		\coordinate (r2) at (2.6,1.7);
		\coordinate (r3) at (3.2,2.7);
		\coordinate (r4) at (3.7,3.8);
		
		\coordinate (r6) at (4.4,5.7);
		\coordinate (r7) at (5.1,6.5);

		\coordinate (a) at (6,7);
		\coordinate (b) at (7,7);
		
		\coordinate (s1) at (7.9,6.5);
		\coordinate (s2) at (8.6,5.7);
		\coordinate (s3) at (9.0,4.6);
		\coordinate (s4) at (9.5,3.3);
		
		\coordinate (s6) at (10.4,1.7);
		\coordinate (s7) at (11.2,0.8);

		\coordinate (w') at (12.1,0);
		\coordinate (z') at (13,0);
		
		\coordinate (u1) at (0.7,3.3);
		\coordinate (u2) at (1.8,5.7);
		\coordinate (u9) at (2.9,8.1);
		
		\coordinate (v1) at (10.1,8.1);
		\coordinate (v2) at (11.2,5.7);
		\coordinate (v9) at (12.3,3.3);

		\begin{pgfonlayer}{front}
		
		\foreach \i in {x', y', w', z', a, b, r1, r2, r3, r4, r6, r7, s1, s2, s3, s4, s6, s7}
		\fill  (\i) circle (2pt);
		
		\foreach \i in {u1, u2, u9, v1, v2, v9}{
			\draw[blue!75!black, very thick]  (\i) circle (2pt);
			\fill[blue!75!white]  (\i) circle (2pt);
		}
		
		\node at (0,-0.4) {\textcolor{black!60!black}{$x'$}};
		\node at (0.9,-0.43) {\textcolor{black!60!black}{$y'$}};
		\node at (6,6.58) {
			{$a$}};
		\node at (7,6.63) {
			{$b$}};
		\node at (12.1,-0.4) {\textcolor{black!60!black}{$w'$}};
		\node at (13,-0.4) {\textcolor{black!60!black}{$z'$}};
		
		\node at (2.1,0.5) {$r_1$};
		\node at (10.9,0.5) {$s_{m}$};
		
		\node at (2.9,1.4) {$r_2$};
		\node at (10,1.4) {$s_{m-1}$};
		
		\node at (4.9,5.5) {$r_{m-1}$};
		\node at (8.3,5.5) {$s_{2}$};
		
		\node at (5.6,6.2) {$r_{m}$};
		\node at (7.6,6.2) {$s_{1}$};
		
		\node at (3.6,2.5) {$r_3$};
		\node at (4.1,3.7) {$r_4$};
		
		\node at (8.7,4.4) {$s_3$};
		\node at (9.2,3.1) {$s_4$};
		
		\node at (1.0,6.5) {\textcolor{blue!75!black}{$U_{x'y'}$}};
		\node at (0.85,3.7) {\footnotesize \textcolor{blue!75!black}{$u_{i(1)}$}};
		\node at (1.95,6.1) {\footnotesize \textcolor{blue!75!black}{$u_{i(2)}$}};
		\node at (1.9,8.15) {\footnotesize \textcolor{blue!75!black}{$u_{i\left(\frac{m}{2}+1\right)}$}};
		
		\node at (12.0,6.5) {\textcolor{blue!75!black}{$U_{w'z'}$}};
		\node at (10.9,8.25) {\footnotesize \textcolor{blue!75!black}{$v_{j(1)}$}};
		\node at (11.25,6.1) {\footnotesize \textcolor{blue!75!black}{$v_{j(2)}$}};
		\node at (12.1,3.75) {\footnotesize \textcolor{blue!75!black}{$v_{j\left(\frac{m}{2}+1\right)}$}};
		
		\end{pgfonlayer}
		
		\begin{pgfonlayer}{background}
		\draw[black!40!white, line width=2pt] (x') -- (y');
		\draw[black!40!white, line width=2pt] (a) -- (b);
		\draw[black!40!white, line width=2pt] (w') -- (z');
		
		\draw[black!40!white, line width=2pt] (y') -- (r1);
		\draw[black!40!white, line width=2pt] (r1) -- (r2);
		\draw[black!40!white, line width=2pt] (r2) -- (r3);
		\draw[black!40!white, line width=2pt] (r3) -- (r4);
		
		\draw[black!40!white, line width=2pt, decorate, decoration={snake,segment length=9,amplitude=4}] (r4) -- (r6);
		
		\draw[black!40!white, line width=2pt] (r6) -- (r7);
		\draw[black!40!white, line width=2pt] (r7) -- (a);
		
		\draw[black!40!white, line width=2pt] (b) -- (s1);
		\draw[black!40!white, line width=2pt] (s1) -- (s2);
		\draw[black!40!white, line width=2pt] (s2) -- (s3);
		\draw[black!40!white, line width=2pt] (s3) -- (s4);
		
		\draw[black!40!white, line width=2pt, decorate, decoration={snake,segment length=9,amplitude=4}] (s4) -- (s6);
		
		\draw[black!40!white, line width=2pt] (s6) -- (s7);
		\draw[black!40!white, line width=2pt] (s7) -- (w');
		\end{pgfonlayer}
		
		\draw[rotate around={-26:(u2)},blue!75!black, line width=2pt] (u2) ellipse (18pt and 3.4cm);
		\fill[blue!75!white,opacity=0.2,rotate around={-26:(u2)}] (u2) ellipse (18pt and 3.4cm);
		
		\draw[rotate around={26:(v2)},blue!75!black, line width=2pt] (v2) ellipse (18pt and 3.4cm);
		\fill[blue!75!white,opacity=0.2,rotate around={26:(v2)}] (v2) ellipse (18pt and 3.4cm);
		
		\qedge{(u1)}{(y')}{(x')}{4.5pt}{1.5pt}{red!70!black}{red!70!black,opacity=0.2};
		\qedge{(u1)}{(r1)}{(y')}{4.5pt}{1.5pt}{red!70!black}{red!70!black,opacity=0.2};
		\qedge{(u1)}{(r2)}{(r1)}{4.5pt}{1.5pt}{red!70!black}{red!70!black,opacity=0.2};
		
		\qedge{(u2)}{(r2)}{(r1)}{4.5pt}{1.5pt}{red!70!black}{red!70!black,opacity=0.2};
		\qedge{(u2)}{(r3)}{(r2)}{4.5pt}{1.5pt}{red!70!black}{red!70!black,opacity=0.2};
		\qedge{(u2)}{(r4)}{(r3)}{4.5pt}{1.5pt}{red!70!black}{red!70!black,opacity=0.2};
		
		\qedge{(u9)}{(r7)}{(r6)}{4.5pt}{1.5pt}{red!70!black}{red!70!black,opacity=0.2};
		\qedge{(u9)}{(a)}{(r7)}{4.5pt}{1.5pt}{red!70!black}{red!70!black,opacity=0.2};
		\qedge{(u9)}{(b)}{(a)}{4.5pt}{1.5pt}{red!70!black}{red!70!black,opacity=0.2};
		
		\qedge{(v1)}{(b)}{(a)}{4.5pt}{1.5pt}{red!70!black}{red!70!black,opacity=0.2};
		\qedge{(v1)}{(s1)}{(b)}{4.5pt}{1.5pt}{red!70!black}{red!70!black,opacity=0.2};
		\qedge{(v1)}{(s2)}{(s1)}{4.5pt}{1.5pt}{red!70!black}{red!70!black,opacity=0.2};
		
		\qedge{(v2)}{(s2)}{(s1)}{4.5pt}{1.5pt}{red!70!black}{red!70!black,opacity=0.2};
		\qedge{(v2)}{(s3)}{(s2)}{4.5pt}{1.5pt}{red!70!black}{red!70!black,opacity=0.2};
		\qedge{(v2)}{(s4)}{(s3)}{4.5pt}{1.5pt}{red!70!black}{red!70!black,opacity=0.2};
		
		\qedge{(v9)}{(z')}{(w')}{4.5pt}{1.5pt}{red!70!black}{red!70!black,opacity=0.2};
		\qedge{(v9)}{(w')}{(s7)}{4.5pt}{1.5pt}{red!70!black}{red!70!black,opacity=0.2};
		\qedge{(v9)}{(s7)}{(s6)}{4.5pt}{1.5pt}{red!70!black}{red!70!black,opacity=0.2};
		
		\end{tikzpicture}
		\caption{Idea of the second step, the picture is similar to \cite[Fig. 4.1]{5/9}}
	\end{figure}
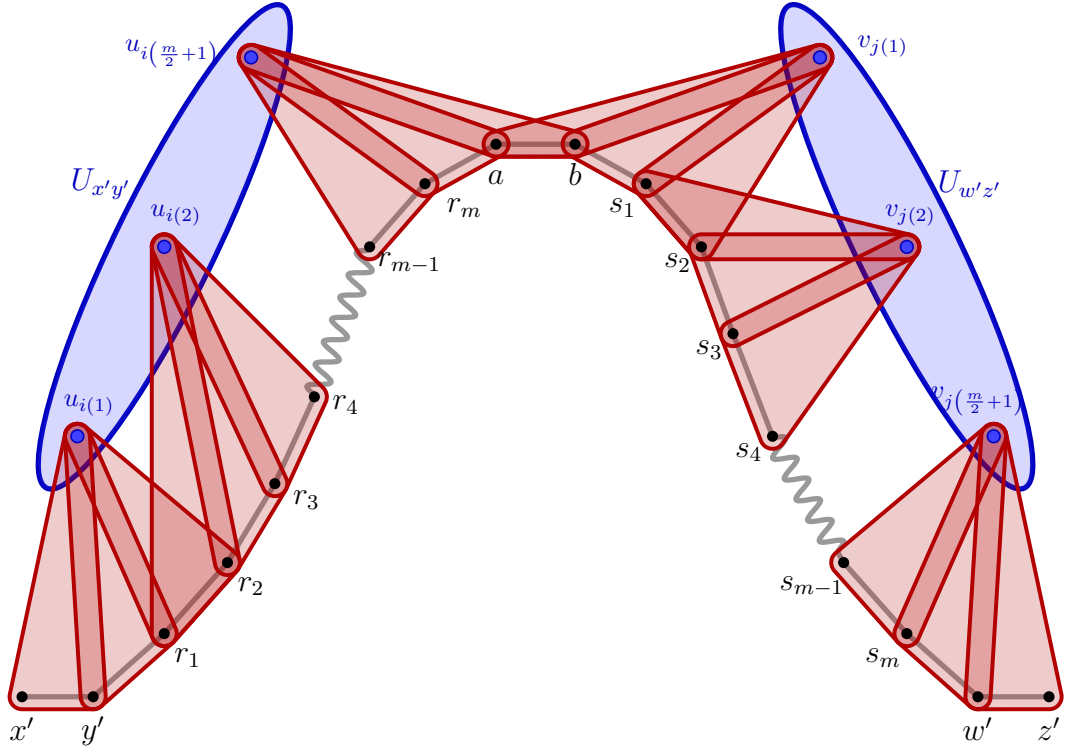
	
	Let $m$ be the smallest even number~${\geq \frac{1}{\alpha} + 1}$.
	It now suffices to show that for some~$\vartheta'>0$ with~$1/n\ll \vartheta ' \ll \alpha$ we have the following.
	For all ordered pairs $(x',y'),(w',z')\in[n]^{2}$ for which the vertices within each pair are distinct and~${x',y',w',z'\geq \frac{n}{2}}$, the number of~$(x',y')$-$(w',z')$ walks with~$3m+4$ internal vertices is at least~${\vartheta ' n^{3m+4}}$.
	
	Since~$d(x',y')\geq\frac{n}{2} + \alpha n$, there exists a set~$U_{x'y'}= \left\{u_1, \dots, u_{\alpha n}\right\} \subseteq [n]\setminus [n/2]$ such that~$x'y' \in E\left( L_{u_i} \right)$, for all~$i \in [\alpha n]$ (recall that~$L_{u_i}$ denotes the link graph of~$u_i$).
	Similarly, there exists~$U_{w'z'}=\left\{v_1, \dots, v_{\alpha n}\right\} \subseteq [n]\setminus[n/2]$ such that~$w'z' \in E\left( L_{v_i} \right)$, for all~$i \in [ \alpha n]$.
	
	For $(a,b) \in [n]^2$, let~$I_{ab} = \left\{i\in \left[\alpha n \right] : ab \in E\left(L_{u_i}\right) \cap E\left(L_{v_i}\right) \right\}$. 
	Since all vertices~$\geq \frac{n}{2}$ (apart from~$u_i, v_i$) have in both~$L_{u_i}$ and~$L_{v_i}$ at least~$\frac{n}{2} + \alpha n$ neighbours, and therefore~$2\alpha n$ vertices that they are adjacent to in both~$L_{u_i}$ and~$L_{v_i}$, there are at least~$\frac{\alpha n^2}{4}$ edges in~$L_{v_i}\cap L_{u_i}$. 
	Thus, by double counting we have 
	$$\sum_{\mathclap{(a,b) \in [n]^2}} \vert I_{ab} \vert \geq \sum_{i \in \left[\alpha n \right]} \left \vert E\left(L_{v_i}\right) \cap E\left(L_{u_i}\right) \right \vert \geq \frac{\alpha n^2}{4} \alpha n\,.$$
	
	Next, for fixed $(a,b)\in [n]^2$, we find a lower bound on the number $L_{ab}$ of~$3$-uniform walks of the form
	$$x'y'u_{i(1)}r_1r_2u_{i(2)}\dots u_{i\left(\frac{m}{2}\right)}r_{m-1}r_m u_{i\left(\frac{m}{2}+1\right)}ab$$
	where $y'r_1r_2\dots r_{m-1}r_ma$ is a~$2$-uniform walk in $L_{u_{i(k)}}$ and $i(k) \in I_{ab}$, for all $k \in \left[\frac{m}{2}+1\right]$.
	
	To this goal, first observe that for all~$i\in [ \alpha n]$, the number of~$y'a$-walks of length~$m+1$ in~$L_{u_i}$ is at least~$\left(\frac{\alpha}{3}\right) ^m n^m$.
	Indeed, since~$u_i \geq n/2$, we know that for~$j\in[n]$, we have~$d_{L_{u_i}}(j)\geq \min \left(j, \frac{n}{2}\right)+\alpha n$. 
	Therefore, there are at least $\left(\frac{\alpha n}{2}\right) ^{m-1}$ walks of length~$m-1$ starting in~$a$ in which each vertex is either at least~$ \frac{n}{2} +\frac{\alpha n}{2}$ or at least~$\frac{\alpha n}{2}$ larger than the preceding vertex.
	Since we set~${m\geq 1/\alpha +1}$, each of these walks ends in a vertex $\geq \frac{n}{2}$ and for at least~$\left(\frac{\alpha n}{3}\right)^{m-1}$ of them the last vertex is distinct from~$y'$.
	For each such walk~$T$ with its last vertex~$a_T'\neq y'$, there are~$2\alpha n$ possibilities for common neighbours of~$y'$ and~$a_T'$ (note that the degrees in~$L_{u_i}$ of both~$y'$ and~$a_T'$ are at least~$\frac{n}{2} +\alpha n$).
	In total, that gives us at least~$\left(\frac{\alpha n}{3}\right) ^m$ $y'a$-walks of length~$m+1$ in~$L_{u_i}$.
	
	Now for $\vec r \in [n]^m$, we set $D_{ab}\left(\vec r\right) := \left \{ i \in I_{ab}: y'\vec r a \text{ is a walk in }L_{u_i}\right \}$.
	Again by double counting and by the previous observation we infer
	$$\sum_{\mathclap{\vec r \in [n]^m}} \left\vert D_{ab}\left(\vec r\right)\right\vert=\sum_{\mathclap{i\in I_{ab}}}\left\vert\left\{\vec r \in [n]^m: y'\vec r a \text{ is a walk in }L_{u(i)}\right\}\right\vert \geq \left\vert I_{ab}\right\vert\left(\frac{\alpha}{3}\right)^mn^m.$$
	Note that for each~$\vec r \in [n]^m$ that is a~$y'a$-walk in~$L_{u_{i(k)}}$ for every $k\in \left[\frac{m}{2}+1\right]$, we have that $$x'y'u_{i(1)}r_1r_2u_{i(2)}\dots u_{i\left(\frac{m}{2}\right)}r_{m-1}r_mu_{i\left(\frac{m}{2}+1\right)}ab$$ is a~$3$-uniform~$(x'y')$-$(ab)$-walk of length~$m+\frac{m}{2}+3$ in~$H$. Hence, with Jensen's inequality we derive:
	$$L_{ab} \geq \sum_{\mathclap{\vec{r} \in [n]^m}}\left\vert D_{ab}\left(\vec r\right) \right\vert ^{\frac{m}{2}+1} \geq n^m\left(\sum \frac{1}{n^m}\left\vert D_{ab}\left(\vec r\right)\right\vert\right)^{\frac{m}{2}+1} \geq n^m \left(\left\vert I_{ab} \right\vert \left(\frac{\alpha}{3}\right)^m\right)^{\frac{m}{2}+1}.$$
	We define $R_{ab}$ analogously as the number of~$3$-uniform walks of the form
	$$abv_{j(1)}s_1s_2v_{j(2)}\dots v_{j\left(\frac{m}{2}\right)}s_{m-1}s_m v_{j\left(\frac{m}{2}+1\right)}w'z'\,,$$
	where $bs_1s_2\dots s_{m-1}s_mw'$ is a~$2$-uniform walk in~$L_{v_{j(k)}}$ and~$j(k) \in I_{ab}$, for all~$k \in \left[\frac{m}{2}+1\right]$, and get the same lower bound by an analogous argument.
	
	At last, let~$W$ be the number of~$(x'y')\text{-}(w'z')\text{-walks of length }3m+6\text{ in }H$.
	We apply Jensen's inequality a second time to obtain:
	\begin{align*}
	W\geq & \sum_{\mathclap{(a,b)\in [n]^2}} L_{ab}R_{ab} \\
	\geq & n^{2m}\left(\frac{\alpha}{3}\right)^{m^2+2m} \sum_{\mathclap{(a,b)\in [n]^2}} \left\vert I_{ab} \right\vert ^{m+2} \\
	\geq & n^{2m}\left(\frac{\alpha}{3}\right)^{m^2+2m}n^2\left(\frac{1}{n^2}\frac{\alpha^2 n^3}{4}\right)^{m+2} \\
	\geq & \left(\frac{\alpha}{3}\right) ^{m^2+2m}\left(\frac{\alpha ^2}{4}\right)^{m+2} n^{3m+4} \\
	\geq & \left(\frac{\alpha ^2}{4}\right) ^{m^2+3m+2} n^{3m+4}.
	\end{align*}
	In total, putting together the walks connecting~$(x,y)$ and~$(x',y')$, $(x',y')$ and~$(w',z')$ and~$(w',z')$ and~$(w,z)$ we get that the number of~$(x,y)$-$(w,z)$-walks of length~$2 \cdot\frac{2}{\alpha} + 3m+6$ in~$H$ is at least
	$$\left(\left( \frac{\alpha}{5}\right) ^{\frac{2}{\alpha} } n^{\frac{2}{\alpha}}\right)^2 \times \left(\frac{\alpha ^2}{4}\right) ^{m^2+3m+2} n^{3m+4}\geq \alpha^{m^3}n^{\frac{4}{\alpha} +3m+4}\,.$$
	
	Since only~$\mathcal{O}\left(n^{\frac{4}{\alpha} +3m+3}\right)$ of these fail to be a path, we are done.
\end{proof}

\section{Reservoir}\label{sec_reservoir}
In this section, we will prove the existence of a small set, the reservoir, such that any two pairs of vertices can be connected by paths with all internal vertices lying in the reservoir. The probabilistic proof of this lemma as done in~\cite{5/9} works in almost the same way with different conditions as soon as the Connecting Lemma is provided. We will state two inequalities first that we will need for the probabilistic method.
\begin{lemma}[Chernoff, see for instance Cor. 2.3 in \cite{probabilistic1}]\label{chernoff_ineq}
	Let $X_1, X_2, \dots, X_m$ be a sequence of~$m$ independent random variables $X_i: \rightarrow \{0,1\}$ with $\mathbb{P}\left( X_i = 1\right) = p$ and~${\mathbb{P}\left( X_i = 0\right) = 1-p}$. Then we have for $\delta \in (0,1)$:
	\begin{itemize}
		\item $\mathbb{P}\left( \sum_{i\in [m]} X_i \geq \left( 1+\delta\right) pm\right) \leq \exp\left(-\frac{\delta ^2}{3}pm\right)$
		\item $\mathbb{P}\left( \sum_{i\in [m]} X_i \leq \left( 1-\delta\right) pm\right) \leq \exp\left(-\frac{\delta ^2}{2}pm\right)$
	\end{itemize}
\end{lemma}
\begin{lemma}[Azuma-Hoeffding, McDiarmid, Cor.~2.27 in \cite{probabilistic1} and Thm.~1 in \cite{probabilistic2}]\label{azuma_hoeffding_ineq}
	Suppose that~${X_1, \dots , X_m}$ are independent random variables taking values in~${\Lambda_1, \dots , \Lambda_m}$ and let~${f:\Lambda_1 \times \dots \times \Lambda_m\rightarrow \mathbb{R}}$ be a measurable function. Moreover, suppose that for certain real numbers $c_1, \dots , c_m \geq 0$, we have that
	if~${J,J'\in \prod \Lambda_i}$ differ only in the $k$-th coordinate, then~${\left\vert f(J) - f\left( J'\right)\right\vert \leq c_k}$.
	Then the random variable $X := f\left( X_1, \dots , X_m\right)$ satisfies
	$$\mathbb{P}\left( \left\vert X - \mathbb{E}(X)\right\vert \geq t\right) \leq 2\exp\left(-\frac{2t^2}{\sum c_i^2}\right)$$
\end{lemma}

We are now ready to prove Lemma \ref{reservoir_lemma}.

\begin{proof}[Proof of Lemma \ref{reservoir_lemma}]
	Let $\alpha$,~$L$,~$\vartheta$,~$n$, and~$H$ be given as in the statement. 
	We choose a random subset $\mathcal{R} \subseteq [n]$, where we select each vertex independently with probability $p=\left(1-\frac{1}{10L}\right) \vartheta ^2$.
	Since $\left\vert\mathcal{R}\right\vert$ is now binomially distributed, we can apply Chernoff's inequality (Lemma \ref{chernoff_ineq}) and utilise the hierarchy to obtain
	\begin{align}\label{reservoir_upper_bound}
	\mathbb{P}\left(\left\vert\mathcal{R}\right\vert < \vartheta ^2 n/2\right) \leq \mathbb{P}\left(\left\vert\mathcal{R}\right\vert < \frac{2}{3} \mathbb{E}\left(\mathcal{R}\right)\right) \leq \exp\left(-\frac{\left(1/3\right)^2}{2}pn\right) < 1/3 .
	\end{align}
	We also have $\vartheta ^2 n \geq \left( 1+c(L)\right) \mathbb{E}\left(\left\vert\mathcal{R}\right\vert\right)$ for some small $c(L) \in (0,1)$ not depending on~$n$ and therefore, again by Chernoff we get for large~$n$:
	\begin{align}\label{reservoir_lower_bound}
	\mathbb{P}\left(\left\vert\mathcal{R}\right\vert > \vartheta ^2 n\right) \leq \mathbb{P}\left(\left\vert\mathcal{R}\right\vert \geq \left(1+c(L)\right) \mathbb{E}\left(\mathcal{R}\right)\right) \leq \exp\left(-\frac{c(L)^2}{3}pn\right) < 1/3
	\end{align}
	By Lemma~\ref{connecting_lemma}, we have that for all disjoint ordered pairs of distinct vertices~$(x,y)$ and~$(w,z)$, the number of~$(x,y)$-$(w,z)$-paths of length~$L$ in~$H$ is at least~$\vartheta n^{L-2}$. 
	Let~$X=X\left((x,y),(w,z)\right)$ denote the random variable counting the number of those~$(x,y)$-$(w,z)$-paths in~$H$ that are of length~$L$ and have all internal vertices in~$\mathcal{R}$.
	We then have~$\mathbb{E}(X) \geq p^{L-2} \vartheta n^{L-2}$.
	
	Now we apply the Azuma-Hoeffding inequality (Lemma \ref{azuma_hoeffding_ineq}) (with $X_1,\dots , X_n$ being the indicator variables for the events ``$1\in \mathcal{R}$'',\dots ,``$n\in \mathcal{R}$'') which gives us, since the presence or absence of one particular vertex in $\mathcal{R}$ affects $X$ by at most~${\left(L-2\right) n^{L-3}}$, that
	\begin{align*}
	\mathbb{P}\left(X \leq \frac{2}{3} \vartheta (pn)^{L-2}\right) \leq & \mathbb{P}\left( X\leq \frac{2}{3} \mathbb{E}(X)\right) \\
	\leq & 2\exp\left( -\frac{2 \left( p^{L-2}\vartheta n^{L-2}\right) ^2}{9n\left( (L-2)n^{L-3}\right) ^2}\right) \\
	= &\exp\left( -\Omega (n)\right) .
	\end{align*}
	By the union bound, also the probability that there is a pairs of pairs for which the respective number of connecting paths with all internal vertices in~$\mathcal{R}$ is less than~${\frac{2}{3} \vartheta (pn)^{L-2}}$ can be bounded from above by 
	\begin{align}\label{reservoir_many_paths}
	\exp\left(-\Omega (n)\right) \times n^4 < 1/3
	\end{align}
	for~$n$ large.
	Moreover, recalling our hierarchy we have $$\frac{2}{3}\vartheta p^{L-2} n^{L-2} = \left(1-\frac{1}{10 L}\right)^{L-2} \frac{2}{3}\vartheta \left(\vartheta ^2 n\right)^{L-2}\geq \frac{\vartheta}{2}\left(\vartheta ^2 n\right) ^{L-2}$$ which together with (\ref{reservoir_lower_bound}) and (\ref{reservoir_many_paths}) implies the following: With probability $>1/3$ the chosen set~$\mathcal{R}$ satisfies $\vert\mathcal{R}\vert\leq \vartheta ^2 n$ and has the property that for all disjoint ordered pairs of distinct vertices $(x,y)$ and $(w,z)$ there exist at least~${\frac{\vartheta}{2}\left\vert\mathcal{R}\right\vert ^{L-2}}$ paths of length~$L$ in~$H$ that connect those pairs and have all their internal vertices in $\mathcal{R}$.
	Therefore, combining this with~(\ref{reservoir_upper_bound}) ensures that there indeed exists a version of~$\mathcal{R}$ that has all the required properties of our reservoir set.
\end{proof}

It is not hard now to show the preservation of the reservoir, Lemma \ref{reservoir_preservation_lemma}.
\begin{proof}[Proof of Lemma \ref{reservoir_preservation_lemma}]
	Let~$H, \mathcal{R}, \mathcal{R}'$ be as in the statement of the Lemma.
	Consider any two disjoint ordered pairs of distinct vertices $(x,y)$ and $(w,z)$.
	We have $$\left\vert\mathcal{R}'\right\vert \leq 2\vartheta ^4 n \leq \vartheta ^{3/2} \frac{\vartheta^2}{2} n \leq \vartheta ^{3/2} \left\vert \mathcal{R}\right\vert$$ by the lower bound we get from Lemma \ref{reservoir_lemma}.
	Since every particular vertex in~$\mathcal{R}'$ is an internal vertex of at most~${(L-2) \vert\mathcal{R}\vert ^{L-3}}$ of the~$(x,y)$-$(w,z)$-paths of length~$L$ in~$H$ with all internal vertices from~$\mathcal{R}$, the Reservoir Lemma tells us that there are at least
	$$\frac{\vartheta }{2}\vert\mathcal{R}\vert ^{L-2} - \left\vert\mathcal{R}'\right\vert (L-2)\vert\mathcal{R}\vert ^{L-3} \geq \frac{\vartheta }{2}\vert\mathcal{R}\vert ^{L-2} - \vartheta ^{3/2} (L-2) \vert\mathcal{R}\vert ^{L-2} > 0$$
	such $(x,y)$-$(w,z)$-paths with all internal vertices in $\mathcal{R}\setminus\mathcal{R}'$.
\end{proof}

\section{Absorbing Path}\label{sec_absorbing_path}
In this section, we will construct a short (absorbing) path~$P_A$ that can ``absorb'' every small set of arbitrary vertices: For each small set~$X\subseteq V$, we can build a path~$P_A'$ with~$V(P_A')=V(P_A)\cup X$ which has the the same end-pairs as~$P_A$.
Later, it will then suffice to find a cycle containing~$P_A$ and almost all vertices, and subsequently absorb the remaining vertices into~$P_A$.
Since we already have a Connecting Lemma, actually the only step left will be to find a long path.

In order to construct such an absorbing path, one first has to find many \textit{absorbers} for each vertex~$v$: In our case, an absorber is a ``cascade'' of small paths that allows us to build a new such cascade of paths with the same end-pairs, containing all vertices of the first two paths and in addition the ``absorbed'' vertex~$v$ (see Definition~\ref{def: absorber}).
This makes sure that we can maintain the path structure of~$P_A$ when absorbing a vertex since the linking pairs remain unchanged.
Once we know that for every vertex~$v$, there exist many such~$v$-absorbers in~$H$, the probabilistic method provides a small set of disjoint paths with the property that for every vertex~$v$, this set contains many~$v$-absorbers.
Lastly, we will simply connect all these paths via the Connecting Lemma and note that then we can absorb a small set of vertices by greedily inserting each vertex into a different absorber.

To construct the absorbers, we again utilize that we can ``climb up'' the degree sequence.
More precisely, we define the following ``absorbers''.
\begin{dfn}\label{def: absorber}
	Let~$\alpha>0$,~$n\in\mathds{N}$, and~$H=([n],E)$ a~$3$-graph and set~$s=s(\alpha)=2\cdot\frac{1}{\alpha}$.\footnote{Recall that in our convention~$\frac{1}{\alpha}$ is an integer and, hence,~$s$ is an even integer.}
	For~$x \in [n]$, a~$4s$-tuple $$(v_1,w_1,y_1,z_1,\dots,v_s,w_s,y_s,z_s) \in [n]^{4s}$$ of distinct vertices is called~$(x,\alpha)$-$absorber$ (in~$H$) if
	\begin{enumerate}
		\item~$v_1w_1xy_1z_1$ is a path in~$H$,
		\item for~$i\in[s-1]$, we know that~$v_iw_iy_{i+1}z_{i+1}$ and~$v_{i+1}w_{i+1}x_iy_i$ are paths in~$H$, and
		\item~$v_sw_sy_sz_s$ is a path in~$H$.
	\end{enumerate}
\end{dfn}

When~$\alpha$ is not important, we omit it in the notation, then simply speaking of~$x$-absorbers.
Note that we can absorb~$x$ into an~$x$-absorber~$(v_1,w_1,y_1,z_1,\dots,v_s,w_s,y_s,z_s)$ as follows, see also Figure~\ref{fig: abs}.
Before absorption, we consider the paths~$v_iw_iy_{i+1}z_{i+1}$ and~$v_{i+1}w_{i+1}y_iz_i$, for all \emph{odd}~$i\in[s]$.
After absorption, we consider the path~$v_1w_1xy_1z_1$, the paths~$v_iw_iy_{i+1}z_{i+1}$ and~$v_{i+1}w_{i+1}y_iz_i$ for all \emph{even}~$i\in[s-2]$, and the path~$v_sw_sy_sz_s$.
Note that the (ordered) end-pairs of the considered paths are the same before and after absorption.
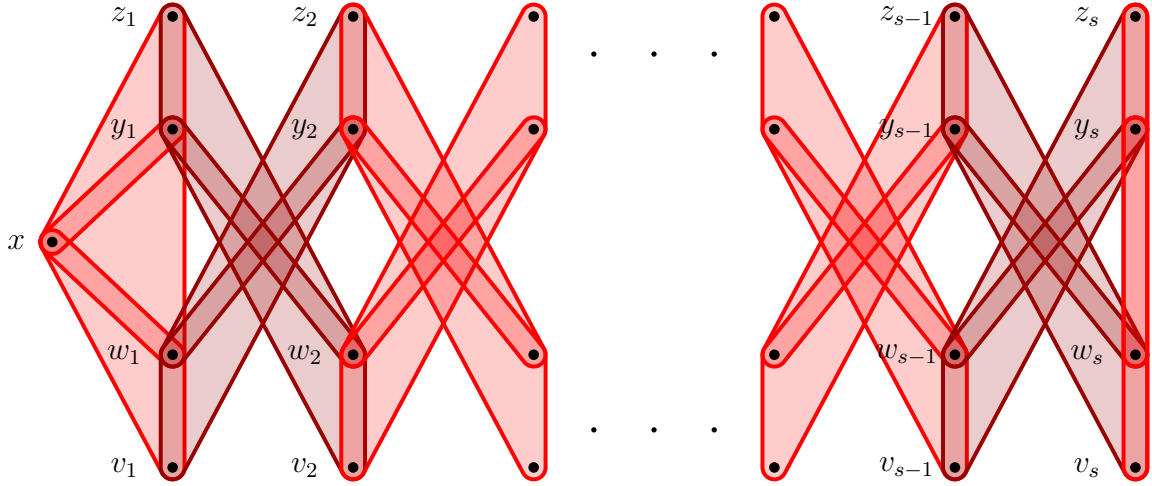
\begin{figure}\label{fig: abs}
	\centering
	\begin{tikzpicture}[x=1.6cm]
	
	\coordinate (x) at (-5,0);
	
	\coordinate (v1) at (-4,-3);	
	\coordinate (w1) at (-4,-1.5);
	\coordinate (y1) at (-4,1.5);
	\coordinate (z1) at (-4,3);
	
	\coordinate (v2) at (-2.5,-3);	
	\coordinate (w2) at (-2.5,-1.5);
	\coordinate (y2) at (-2.5,1.5);
	\coordinate (z2) at (-2.5,3);
	
	\coordinate (v3) at (-1,-3);	
	\coordinate (w3) at (-1,-1.5);
	\coordinate (y3) at (-1,1.5);
	\coordinate (z3) at (-1,3);
	
	\coordinate (d1) at (-0.5,-2.5);
	\coordinate (d2) at (0,-2.5);
	\coordinate (d3) at (0.5,-2.5);

	\coordinate (d1') at (-0.5,2.5);
	\coordinate (d2') at (0,2.5);
	\coordinate (d3') at (0.5,2.5);
	
	\coordinate (vs-2) at (1,-3);	
	\coordinate (ws-2) at (1,-1.5);
	\coordinate (ys-2) at (1,1.5);
	\coordinate (zs-2) at (1,3);
	
	\coordinate (vs-1) at (2.5,-3);	
	\coordinate (ws-1) at (2.5,-1.5);
	\coordinate (ys-1) at (2.5,1.5);
	\coordinate (zs-1) at (2.5,3);
	
	\coordinate (vs) at (4,-3);	
	\coordinate (ws) at (4,-1.5);
	\coordinate (ys) at (4,1.5);
	\coordinate (zs) at (4,3);
	
	\begin{pgfonlayer}{front}
	
	\foreach \a in {v1,v2,w1,w2,y1,y2,z1,z2,x,vs-1,vs,ws-1,ws,ys-1,ys,zs-1,zs,v3,w3,y3,z3,vs-2,ws-2,ys-2,zs-2}
	\fill  (\a) circle (2pt);
	
	\foreach \a in {d1,d2,d3,d1',d2',d3'}
	\fill  (\a) circle (1pt);
	
	\node at (-5.3,0) {$x$};
	
	\node at (-4.4,-3) {$v_1$};
	\node at (-4.4,-1.5) {$w_1$};
	\node at (-4.4,1.5) {$y_1$};
	\node at (-4.4,3) {$z_1$};
	
	\node at (-2.9,-3) {$v_2$};
	\node at (-2.9,-1.5) {$w_2$};
	\node at (-2.9,1.5) {$y_2$};
	\node at (-2.9,3) {$z_2$};
	
	\node at (2.1,-3) {$v_{s-1}$};
	\node at (2.1,-1.5) {$w_{s-1}$};
	\node at (2.1,1.5) {$y_{s-1}$};
	\node at (2.1,3) {$z_{s-1}$};
	
	\node at (3.6,-3) {$v_s$};
	\node at (3.6,-1.5) {$w_s$};
	\node at (3.6,1.5) {$y_s$};
	\node at (3.6,3) {$z_s$};
	
	\end{pgfonlayer}
	\qedge{(v1)}{(x)}{(w1)}{4.5pt}{1.5pt}{red!100!black}{red!100!black,opacity=0.2};
	\qedge{(w1)}{(x)}{(y1)}{4.5pt}{1.5pt}{red!100!black}{red!100!black,opacity=0.2};
	\qedge{(x)}{(z1)}{(y1)}{4.5pt}{1.5pt}{red!100!black}{red!100!black,opacity=0.2};
	
	\qedge{(v2)}{(y1)}{(w2)}{4.5pt}{1.5pt}{red!60!black}{red!60!black,opacity=0.2};
	\qedge{(w2)}{(y1)}{(z1)}{4.5pt}{1.5pt}{red!60!black}{red!60!black,opacity=0.2};
	\qedge{(v1)}{(w1)}{(y2)}{4.5pt}{1.5pt}{red!60!black}{red!60!black,opacity=0.2};
	\qedge{(w1)}{(z2)}{(y2)}{4.5pt}{1.5pt}{red!60!black}{red!60!black,opacity=0.2};
	
	\qedge{(v2)}{(w2)}{(y3)}{4.5pt}{1.5pt}{red!100!black}{red!100!black,opacity=0.2};
	\qedge{(w2)}{(z3)}{(y3)}{4.5pt}{1.5pt}{red!100!black}{red!100!black,opacity=0.2};
	\qedge{(v3)}{(y2)}{(w3)}{4.5pt}{1.5pt}{red!100!black}{red!100!black,opacity=0.2};
	\qedge{(w3)}{(y2)}{(z2)}{4.5pt}{1.5pt}{red!100!black}{red!100!black,opacity=0.2};
	
	\qedge{(vs-2)}{(ws-2)}{(ys-1)}{4.5pt}{1.5pt}{red!100!black}{red!100!black,opacity=0.2};
	\qedge{(ws-2)}{(zs-1)}{(ys-1)}{4.5pt}{1.5pt}{red!100!black}{red!100!black,opacity=0.2};
	\qedge{(vs-1)}{(ys-2)}{(ws-1)}{4.5pt}{1.5pt}{red!100!black}{red!100!black,opacity=0.2};
	\qedge{(ws-1)}{(ys-2)}{(zs-2)}{4.5pt}{1.5pt}{red!100!black}{red!100!black,opacity=0.2};
	
	\qedge{(vs-1)}{(ws-1)}{(ys)}{4.5pt}{1.5pt}{red!60!black}{red!60!black,opacity=0.2};
	\qedge{(ws-1)}{(zs)}{(ys)}{4.5pt}{1.5pt}{red!60!black}{red!60!black,opacity=0.2};
	\qedge{(ys-1)}{(zs-1)}{(ws)}{4.5pt}{1.5pt}{red!60!black}{red!60!black,opacity=0.2};
	\qedge{(ys-1)}{(ws)}{(vs)}{4.5pt}{1.5pt}{red!60!black}{red!60!black,opacity=0.2};		
	\qedge{(vs)}{(ws)}{(ys)}{4.5pt}{1.5pt}{red!100!black}{red!100!black,opacity=0.2};
	\qedge{(ws)}{(ys)}{(zs)}{4.5pt}{1.5pt}{red!100!black}{red!100!black,opacity=0.2};
	
	\end{tikzpicture}
	\caption{Structure of the absorbers with hyperedges used \textcolor{red!60!black}{before absorption} of~$v$ in dark red and hyperedges used \textcolor{red!100!black}{after absorption} of~$v$ in light red.}\label{fig: abs}
\end{figure}

\begin{lemma}[Many Absorbers]\label{absorber_lemma}
	Let~$1/n\ll\vartheta\ll\alpha\ll 1$.
	If $H=([n],E)$ is a~$3$-graph with~${d(i,j) \geq \min\left(i,j, \frac{n}{2}\right) + \alpha n}$ for all~$ij \in [n]^{(2)}$ and~$\mathcal{R}$ is a reservoir set given by Lemma \ref{reservoir_lemma}, then for every~$x\in [n]$, the number of~$(x,\alpha)$-absorbers in~$\left([n] \setminus \mathcal{R} \right) ^{4s(\alpha)}$ is at least~$(\frac{\alpha n}{3})^{4s(\alpha)}$.
\end{lemma}
\begin{proof}[Proof of Lemma \ref{absorber_lemma}]
	Let~$1/n\ll\vartheta\ll\alpha\ll 1$, let~$H$ be as in the statement, and let~$x\in [n]$.
	There are at least $\frac{n}{3}$ possibilities to choose a vertex~$w_1\in [n]\setminus(\mathcal{R}\cup\{x\})$ with~$w_1 \geq \min(x+\frac{\alpha n}{2},\frac{n}{2})$.
	Then, there are at least~$\frac{\alpha n}{3}$ choices for a vertex~$v_1\in N(w_1,x)\setminus\mathcal{R}$ with~$v_1\geq\min(x+\frac{\alpha n}{2},\frac{n}{2})$ since~$\vert N(w_1,x)\vert\geq \min(w_1,x,\frac{n}{2})+\alpha n$ and~$w_1\geq\min(x+\frac{\alpha n}{2},\frac{n}{2})$.
	Similarly, there are at least~$\frac{\alpha n}{3}$ choices for a vertex~$y_1\in N(w_1,x)\setminus(\mathcal{R}\cup\{v_1\})$ with~$y_1\geq\min(x+\frac{\alpha n}{2},\frac{n}{2})$ and at least~$\frac{\alpha n}{3}$ choices for a vertex~$z_1\in N(x,y_1)\setminus(\mathcal{R}\cup\{v_1,w_1\})$ with~$z_1\geq\min(x+\frac{\alpha n}{2},\frac{n}{2})$.
	
	Now assume that for some~$i\in[s-2]$, vertices~$v_j$,~$w_j$,~$y_j$, and~$z_j$ have already been selected, for all~$j\in[i]$, in such a way that all edges required by Definition~\ref{def: absorber} are present and~$v_j,w_j,y_j,z_j\geq\min(x+j\frac{\alpha n}{2},\frac{n}{2})$ for all~$j\in[i]$, and denote the set containing all these vertices, all vertices from~$\mathcal{R}$, and~$x$ by~$A_i$.
	Note that for all~$i\in[s-2]$, we have~$\vert A_i\vert\leq \frac{\alpha n}{7}$.
	Therefore, there are at least~$\frac{\alpha n}{3}$ choices for a vertex~$w_{i+1}\in N(y_i,z_i)\setminus A_i$ with~$w_{i+1}\geq\min(x+(i+1)\frac{\alpha n}{2},\frac{n}{2})$.
	Further, there are at least~$\frac{\alpha n}{3}$ choices for a vertex~$v_{i+1}\in N(w_{i+1},y_i)\setminus A_i$ with~$v_{i+1}\geq\min(x+(i+1)\frac{\alpha n}{2},\frac{n}{2})$.
	Similarly, there are at least~$\frac{\alpha n}{3}$ choices for a vertex~$y_{i+1}\in N(v_i,w_i)\setminus (A_i\cup\{v_{i+1},w_{i+1}\})$ with~$y_{i+1}\geq\min(x+(i+1)\frac{\alpha n}{2},\frac{n}{2})$ and at least~$\frac{\alpha n}{3}$ choices for a vertex~$z_{i+1}\in N(w_i,y_{i+1})\setminus(A_i\cup\{v_{i+1},w_{i+1}\})$ with~$z_{i+1}\geq\min(x+(i+1)\frac{\alpha n}{2},\frac{n}{2})$.
	
	Assume that~$v_j$,~$w_j$,~$y_j$, and~$z_j$ have been selected for all~$j\in[s-1]$ such that all edges required by Definition~\ref{def: absorber} are present and~$v_j,w_j,y_j,z_j\geq\min(x+j\frac{\alpha n}{2},\frac{n}{2})$, for all~$j\in[s-1]$, and denote the set containing all these vertices, all vertices from~$\mathcal{R}$, and~$x$ by~$A_{s-1}$.
	Then there are at least~$\frac{\alpha n}{3}$ choices for a vertex~$w_{s}\in N(y_{s-1},z_{s-1})\setminus A_{s-1}$ with~$w_s\geq\min(x+s\frac{\alpha n}{2},\frac{n}{2})$ and at least~$\frac{\alpha n}{3}$ choices for a vertex~$y_s\in N(v_{s-1},w_{s-1})\setminus (A_{s-1}\cup\{w_s\})$ with~$y_s\geq\min(x+s\frac{\alpha n}{2},\frac{n}{2})$.
	Note that by the choice of~$s$ we have~$v_{s-1},w_{s-1},y_{s-1},z_{s-1},w_s,y_s\geq \min((s-1)\frac{\alpha n}{2},\frac{n}{2})=\frac{n}{2}$.
	Thus, we know that $$\vert N(w_{s},y_{s-1})\cap N(w_s,y_s)\vert\geq\frac{n}{2}+\alpha n+\frac{n}{2}+\alpha n-n\geq2\alpha n$$ and so there are at least~$\alpha n$ choices for~$v_s\in N(w_{s},y_{s-1})\cap N(w_s,y_s)\setminus A_{s-1}$ and, similarly, we know that there are at least~$\alpha n$ choices for~$z_s\in N(w_{s-1},y_{s})\cap N(w_s,y_s)\setminus (A_{s-1}\cup\{v_s\})$.
	
	Observe that if the vertices~$v_1,w_1,y_1,z_1,\dots,v_s,w_s,y_s,z_s$ are chosen in the respective neighbourhoods as described above, they form an~$(x,\alpha)$-absorber.
	Hence, the number of~$(x,\alpha)$-absorbers is indeed at least~$(\frac{\alpha}{3}n)^{4s(\alpha)}$.
\end{proof}

We are now ready to prove Lemma \ref{absorbing_path_lemma}.

\begin{proof}[Proof of Lemma \ref{absorbing_path_lemma}]
	The proof proceeds in two steps.
	First, we will use the probabilistic method, showing that with positive probability a randomly chosen set of~$4s$-tuples contains many absorbers for every vertex while being not too large.
	In the second, part we connect all those paths using the Connecting Lemma.
	
	Let~$1/n\ll\vartheta \ll\alpha$, let~$L\in\mathds{N}$ be given by the Connecting Lemma, let~$s=s(\alpha)$, and let~$H,\mathcal{R}$ be given as in the statement.
	
	Let~$\mathcal{X} \subseteq \left([n] \setminus\mathcal{R}\right)^{4s}$ be a random selection in which each~$4s$-tuple in~$\left([n] \setminus\mathcal{R}\right)^{4s}$ is included independently with probability~$p:=\frac{\vartheta ^23^{4s+2}}{\alpha ^{4s}n^{4s-1}}$.
	Then $\mathbb{E}\left[\left\vert\mathcal{X}\right\vert\right] \leq pn^{4s} = \frac{\vartheta ^23^{4s+2}}{\alpha ^{4s}}n$ and by Markov's inequality we get 
	\begin{align}\label{upper_bound_tuple_number}
	\mathbb{P}\left(\left\vert\mathcal{X}\right\vert > 2 \frac{\vartheta ^23^{4s+2}}{\alpha ^{4s}}n\right) \leq \frac{1}{2} .
	\end{align}
	
	Calling two distinct~$4s$-tuples \emph{overlapping} if they contain a common vertex, we observe that there are at most $(4s)^2n^{8s-1}$ ordered pairs of overlapping~$4s$-tuples.
	Let us denote the number of overlapping pairs with both of their tuples occurring in $\mathcal{X}$ by $D$.
	We then get~${\mathbb{E}[D] \leq (4s)^2n^{8s-1}p^2=(4s)^2\big(\frac{\vartheta ^23^{4s+2}}{\alpha ^{4s}}\big)^2 n}$ and Markov yields 
	\begin{align}\label{overlapping_pairs_upper_bound}
	{\mathbb{P}\left[D>\vartheta ^2n\right]\leq\mathbb{P}\left[D>64s^2\left(\frac{\vartheta ^23^{4s+2}}{\alpha ^{4s}}\right) ^2 n\right] \leq \frac{1}{4}}
	\end{align}
	since~$1/n\ll\vartheta\ll\alpha$.
	
	Next, we focus on the number of absorbers contained in~$\mathcal{X}$.
	For~$x\in [n]$, let~$A_x$ denote the set of all~$(x,\alpha)$-absorbers.
	Lemma \ref{absorber_lemma} gives that for every $x\in [n]$, $$\mathbb{E}\left[\left\vert A_x \cap \mathcal{X}\right\vert\right] \geq \left(\frac{\alpha n}{3}\right)^{4s}p=9\vartheta ^2 n .$$
	Since $\left\vert A_x \cap \mathcal{X} \right\vert$ is binomially distributed, we may apply Chernoff's inequality to get for every~$x\in[n]$,
	\begin{align}\label{selected_absorbers_number}
		\mathbb{P}\left( \left\vert A_x \cap \mathcal{X} \right\vert \leq 3\vartheta ^2 n\right) \leq \exp \left(-\frac{\left(\frac{2}{3}\right)^2}{2}9\vartheta ^2 n\right) < \frac{1}{5n}\,.
	\end{align}
	Hence, by the union bound and (\ref{upper_bound_tuple_number}), (\ref{overlapping_pairs_upper_bound}) and (\ref{selected_absorbers_number}), there exists a selection $\mathcal{F}_* \subseteq\left([n] \setminus\mathcal{R}\right)^{4s}$ with:
	\begin{itemize}
		\item $\left\vert \mathcal{F}_* \right\vert \leq \frac{2\vartheta ^23^{4s+2}}{\alpha ^{4s}}n$
		\item $\mathcal{F}_*$ contains at most $\vartheta ^2 n$ overlapping pairs
		\item $\mathcal{F}_*$ contains at least $3 \vartheta ^2 n$~$x$-absorbers, for every $x\in [n]$
	\end{itemize}
	For each overlapping pair, we delete one of its~$4s$-tuples and thus, for every~$x\in [n]$, we lose at most~$\vartheta ^2n$~$x$-absorbers.
	Furthermore, we delete every~$4s$-tuple~$A\in\mathcal{F}_*$ for which there does not exist an~$x\in[n]$ such that~$A$ is an~$x$-absorber.
	Note that now every remaining tuple has all edges present as in Definition~\ref{def: absorber} and all its vertices are distinct. 
	This deletion process gives rise to an~${\mathcal{F} \subseteq \left([n]\setminus \mathcal{R}\right) ^{4s}}$ satisfying:
	\begin{itemize}
		\item $\left\vert \mathcal{F} \right\vert \leq \frac{2\vartheta ^23^{4s+2}}{\alpha ^{4s}}n$,
		\item for every~$4s$-tuple~$A\in\mathcal{F}$ there is an~$x\in [n]$ such that~$A$ is an~$x$-absorber, in particular, all the vertices in~$A$ are distinct and there are edges present as in Definition~\ref{def: absorber}, and
		\item for every~$x \in [n]$, there are at least~$2\vartheta ^2n$~$x$-absorbers in~$\mathcal{F}$.
	\end{itemize}

	Next, we want to connect the elements in $\mathcal{F}$ to a path utilising the Connecting Lemma.
	Let $\mathcal{G}$ be the set consisting of all the paths~$v_iw_iy_{i+1}z_{i+1}$ and~$v_{i+1}w_{i+1}y_iz_i$ for odd~$i$ and for each $\left(v_1,w_1,y_1,z_1,\dots,v_s,w_s,y_s,z_s\right) \in \mathcal{F}$: 
	$$\mathcal{G} = \bigcup_{\left(v_1,w_1,y_1,z_1,\dots,v_s,w_s,y_s,z_s\right) \in \mathcal{F}}\big\{v_{i+j}w_{i+j}y_{i+1-j}z_{i+1-j}:i\in[s]\text{ odd},j\in\{0,1\}\big\}$$
	We then have~$\vert \mathcal{G} \vert = 2\vert\mathcal{F}\vert \leq \frac{4\vartheta ^23^{4s+2}}{\alpha ^{4s}}n$.	
	Let $\mathcal{G}^* \subseteq \mathcal{G}$ be a maximal subset such that there exists a path~${P^* \subseteq H-\mathcal{R}}$ with:
	\begin{itemize}
		\item $P^*$ contains all paths in~$\mathcal{G}^*$ as subpaths
		\item $V\left(P^*\right)\cap \bigcup_{P\in\mathcal{G}\setminus\mathcal{G}^*}V(P) = \emptyset$
		\item $P^*$ satisfies~${v\left(P^*\right) = \left(L+2\right) (\left\vert\mathcal{G}^*\right\vert-1) +4}$.
	\end{itemize}
	First assume $\mathcal{G}^* \subsetneq \mathcal{G}$, and let~${Q^* \in \mathcal{G}\setminus\mathcal{G}^*}$.
	Notice that recalling~$1/n\ll\vartheta\ll\alpha, 1/L\ll 1$, we have
	\begin{align}\label{path_and_reservoir_lemma}
	v\left( P^*\right) + \Big\vert \bigcup_{P\in\mathcal{G}\setminus\mathcal{G}^*}V(P)\Big\vert + \left\vert \mathcal{R}\right\vert \leq \left( L+2\right) \frac{4\vartheta ^23^{4s+2}}{\alpha ^{4s}}n + \vartheta ^2 n \leq \frac{\vartheta n}{2\left(L-2\right) } .
	\end{align}
	Now Lemma~\ref{connecting_lemma} tells us that there are at least~$\vartheta n^{L-2}$ paths of length~$L$ connecting the ending-pair~$(a,b)$ of~$P^*$ with the starting-pair~$(b,c)$ of~$Q^*$ (which are disjoint by the choice of~$P^*$).
	By (\ref{path_and_reservoir_lemma}), at least half of those are disjoint to~$\mathcal{R}\cup\bigcup_{P\in\mathcal{G}\setminus(\mathcal{G}^*\cup \{Q^*\})}V(P)$ and (apart from the end-pairs)  disjoint to~$V\left(P^*\right)$ and~$V\left( Q^*\right)$.
	Hence, there exists a path~$P^{**}$ starting with~$P^*$ and ending with~$Q^*$ whose vertex set is disjoint to~$\mathcal{R}\cup\bigcup_{P\in\mathcal{G}\setminus(\mathcal{G}^*\cup \{Q^*\})}V(P)$ and for which we further have $$v\left(P^{**}\right) = v\left(P^*\right) + L-2 + v\left(Q^*\right) = 4+\left( L+2\right)(\left\vert \mathcal{G}^* \cup \left\{Q^*\right\}\right\vert-1) \,.$$ 
	Therefore,~$\mathcal{G}^* \cup \left\{Q^*\right\}$ contradicts the maximality of~$\mathcal{G}^*$ and thus,~$\mathcal{G}^* = \mathcal{G}$.
	Further, for~$P_A:=P^*$, the hierarchy~$1/n\ll\vartheta\ll\alpha, 1/L\ll 1$ gives us the required bound on~$v\left(P_A\right)$:
	$$v\left(P_A\right) \leq 4+(L+2)\frac{4\vartheta ^23^{4s+2}}{\alpha ^{4s}}n \leq \vartheta n .$$
	Lastly, the structure and the number of the absorbers in~$P_A$ ensure the absorbing property: Let~$X \subseteq [n]$ with~$\vert X \vert \leq 2 \vartheta ^2 n$.
	For each~$x\in X$, we can choose one~$x$-absorber~$(v_1,w_1,y_1,z_1,\dots,v_s,w_s,y_s,z_s)$ from~$\mathcal{F}$ such that all chosen absorbers are distinct, since for every~$x\in V$, the number of~$x$-absorbers in~$\mathcal{F}$ is at least~$2\vartheta^2 n$.
	For every $x\in X$, we then ``open''~$P_A$ at the paths~$v_{i+j}w_{i+j}y_{i+1-j}z_{i+1-j}$ for~$i\in[s]$ odd and~$j\in\{0,1\}$ and reconnect it to a path containing~$x$ by instead considering the paths~$v_{i+j}w_{i+j}y_{i+1-j}z_{i+1-j}$, for all even~$i\in[s]$ and~$j\in\{0,1\}$, and the paths~$v_1w_1xy_1z_1$ and~$v_sw_sy_sz_s$.
	That leaves us with a path $P'$ which satisfies~$V\left( P'\right) = V\left(P_A\right) \cup X$ and has the same end-pairs as~$P_A$.
\end{proof}

\section{Long Path}\label{sec_long_path}
In this section, we will prove the existence of a path that contains almost all vertices.
To do so, we will need a weak form of the hypergraph regularity method which we will therefore introduce briefly.

Let $H=(V,E)$ be a~$3$-graph and $V_1,V_2,V_3 \subseteq V$; we write $$E\left( V_1,V_2,V_3\right) = \left\{\left(v_1,v_2,v_3\right) \in V_1\times V_2\times V_3 : v_1v_2v_3 \in E\right\}$$ and $e\left( V_1,V_2,V_3\right) = \left\vert E(V_1,V_2,V_3)\right\vert $.
Further, we write $$H(V_1,V_2,V_3)=(V_1\dot\cup V_2\dot\cup V_3,E\left( V_1,V_2,V_3\right)).$$

For $\delta > 0, d\geq 0$ and~$V_1,V_2,V_3 \subseteq V$, we say that~$H(V_1,V_2,V_3)$ is \textit{weakly}~$(\delta , d)$-\textit{quasirandom} if for all $U_1 \subseteq V_1, U_2 \subseteq V_2, U_3 \subseteq V_3$, we have that $$\left\vert e\left( U_1, U_2, U_3\right) - d\left\vert U_1\right\vert \left\vert U_2\right\vert \left\vert U_3\right\vert \right\vert \leq \delta \left\vert V_1\right\vert \left\vert V_2\right\vert \left\vert V_3\right\vert .$$
We say that $H(V_1,V_2,V_3)$ is weakly $\delta$-quasirandom if it is weakly~$\left(\delta , d\right)$-quasirandom for some~${d\geq 0}$.
For brevity, we might also say that~$V_1,V_2,V_3$ are weakly~$\left(\delta , d\right)$-quasirandom (or~$\delta$-quasirandom) (in~$H$).
Lastly, since we only look at weak quasirandomness in this section, we may omit the prefix ``weakly''. 

The regularity lemma is a strong tool in extremal combinatorics. While the full generalisation to hypergraphs is more involved than the version for graphs, there is also a light version for hypergraphs that can already be useful and indeed it is for us:
\begin{lemma}[Weak Hypergraph Regularity Lemma]\label{weak_regularity_lemma}
	For $\delta > 0, t_0 \in \mathds{N}$, there exists a~${T_0 \in \mathds{N}}$ such that for every~$3$-graph $H=\left( [n],E\right)$ with $n\geq t_0$, there exist an integer~$t$ with~$t_0\leq t\leq T_0$ and a partition $[n]=V_0\dot\cup V_1 \dot\cup \dots \dot\cup V_t$ such that:
	\begin{itemize}
		\item $\left\vert V_0 \right\vert \leq \delta n$ and $\left\vert V_1\right\vert = \dots = \left\vert V_t \right\vert$
		\item for $i \geq 1$, we have $\max\left( V_i\right) \leq \max\left(V_{i+1}\right)$ and $\max\left( V_i\right) -\min\left( V_i\right) \leq \frac{n}{t_0}$
		\item there are at most $\delta t^3$ sets $ijk\in [t]^{(3)}$ such that the ``triplet'' $V_i,V_j,V_k$, also written as $V^{ijk}$, is not~$\delta$-quasirandom in~$H$.
	\end{itemize}
\end{lemma}
For a proof of Lemma \ref{weak_regularity_lemma} see for instance \cite{hyper_reg1,hyper_reg2,hyper_reg3}.
One can get the slight extra requirement on the ordering of the vertices by dividing the vertex set in intervals of length~$\frac{n}{t_0}$ and afterwards going on with the proof refining those sets.
This has been remarked before, e.g., by Reiher, R\"odl, and Schacht in \cite{K43minus}.

We will regularise~$H$ and then observe that a quasirandom triplet~$V^{ijk}$ with positive density can almost be covered with not too short disjoint paths. 
Thus, we can think of the situation as a reduced hypergraph with the partition classes as vertices and edges encoding those ``good triplets'' that in~$H$ we can almost cover with paths. 
At that point we will notice that the degree condition can almost be transferred to the reduced hypergraph. 
In Lemma~\ref{matching_lemma}, we prove that this degree condition will ensure an almost perfect matching in the reduced hypergraph. 
But that means that in~$H$ almost all vertices can be covered with paths, which we can then connect through the reservoir to a long path in~$H$.
\begin{lemma}[Good Triplets]\label{good_triplets_lemma}
	For~$\xi > 0, d>0, \delta > 0, n \in \mathds{N}$ with~$\frac{d\xi ^3 -\delta}{2}n \geq 1$, the following holds.
	Let~${H=\left( U\dot\cup V\dot\cup W, E\right)}$ with~$\vert U\vert, \vert V\vert, \vert W\vert = n$ be a~$3$-graph and suppose that~$U,V,W$ are~$\left( \delta , d\right)$-quasirandom in~$H$.
	Then at least~$(1-\xi)3n$ vertices of~$H$ can be covered by vertex-disjoint paths of length at least~$\frac{d\xi ^3 -\delta}{2}n -2$.
\end{lemma}

\begin{proof}[Proof of Lemma \ref{good_triplets_lemma}]
	For convenience set~$c=\frac{d\xi ^3 -\delta}{6}n$.
	Let~$\mathcal{P}$ be a maximal set of vertex-disjoint paths of length~$3 c-2$ in~$H$, where each path takes alternatingly vertices from each partition class, i.e., each path is of the form $$u_1v_1w_1u_2v_2w_2 \dots u_{ c }v_{c }w_{c }$$ with $u_i \in U, v_i \in V, w_i \in W$.
	
	Assume that~$\vert V\vert - \left\vert \bigcup_{P\in\mathcal{P}} V(P) \right\vert > 3\xi n$. Then the sets $$U':=U\setminus \bigcup_{P\in\mathcal{P}} V(P), V':=V\setminus \bigcup_{P\in\mathcal{P}} V(P), W':=W\setminus \bigcup_{P\in\mathcal{P}} V(P)$$ satisfy $\left\vert U'\right\vert, \left\vert V' \right\vert, \left\vert W' \right\vert > \xi n$.
	
	Next, we will delete all the edges that contain vertex pairs of small pair degree.
	With the edges that still remain after this process we can build a path of the required length. 
	
	We start with $F_1=H\left[ U',V',W'\right]$ and set $F_{i+1}$, for $i\geq 1$, as the hypergraph obtained from~$F_{i}$ by deleting all edges containing a  vertex pair $xy$ with $d_{F_i}^{\times}(x,y) \leq  c$, where~${d_{F_i}^{\times}(x,y) = \left\vert\left\{e\in E\left(F_i\right): x,y\in e,\left\vert e\cap U'\right\vert = \left\vert e\cap V'\right\vert = \left\vert e\cap W' \right\vert = 1\right\}\right\vert}$.
	This process stops with a hypergraph~$F_j$ in which for all~$x,y \in V\left( F_j\right)$, we either have~${d_{F_j}^{\times}(x,y)=0}$ or~$d_{F_j}^{\times}(x,y)\geq c$.
	The deletion condition guarantees $$e^{\times}(F_1)-e^{\times}\left( F_j\right) \leq 3cn^2 \,,$$ with~${e^{\times}\left( F_i\right) = \left\vert\left\{e\in E\left(F_i\right): \left\vert e\cap U'\right\vert = \left\vert e\cap V'\right\vert = \left\vert e\cap W' \right\vert = 1\right\}\right\vert}$, and the quasirandomness of~$U,V,W$ gives that~${e^{\times}(F_1)=e\left(U',V',W'\right) \geq \left( d\xi ^3 - \delta\right) n^3}$.
	Thus, there still exists an edge~$u_1v_1w_1$ in $F_j$ with $u_1\in U'$, $v_1\in V'$ and~$w_1 \in W'$.
	But this means that there is a path of length~$3c -2$ in~$F_j$: Let~$P^*=u_1v_1w_1 \dots u_kv_kw_k$ be a maximal path in~$F_j$ with~${u_i \in U', v_i \in V'}$ and~$w_i \in W'$, for all~${i \in [k]}$ (note that $k\geq 1$).
	Assuming~$k< c$ for a contradiction, less than~$c$ vertices of~$U'$ appear in~$P^*$.
	But since $v_kw_k$ is contained in the edge $u_kv_kw_k \in E^{\times}\left( F_j\right)$, we actually have that~${d_{F_j}^{\times}\left(v_k,w_k\right)\geq c}$, whence there is a~${u_{k+1} \in U'\setminus V\left( P^*\right)}$ such that~$P^*u_{k+1}$ is a path in~$F_j$.
	
	The same argument applied to $w_ku_{k+1}$ gives a $v_{k+1} \in V'$ such that $P^*u_{k+1}v_{k+1}$ is a path in $F_j$ and finally applying the argument to $u_{k+1}v_{k+1}$ gives rise to a $w_{k+1} \in W'$ such that the path $P^*u_{k+1}v_{k+1}w_{k+1}$ exists in $F_j$ and thus contradicts the maximality of~$P^*$, telling us that~$P^*$ actually contains an alternating path of length~$3c-2$.
	That, on the other hand, gives us another alternating path of length at least~$3 c -2$ that is vertex-disjoint to all paths in~$\mathcal{P}$ and, therefore, contradicts the maximality of $\mathcal{P}$.
	So we indeed have~${\vert V\vert - \left\vert\bigcup_{P\in\mathcal{P}} V(P)\right\vert \leq 3\xi n}$.
\end{proof}

As mentioned before, we later want to find an almost perfect matching in a reduced hypergraph whose edges represent ``good'' triplets as in Lemma~\ref{good_triplets_lemma}.
Then ``translating back'' those edges in the matching will give us a set of (not too many) paths in~$H$ which almost covers all vertices.
To find an almost perfect matching in a hypergraph satisfying the pair degree condition in Theorem~\ref{main_theorem} for almost all pairs, we look at a maximal matching in which the sum of the vertices not contained in it is also maximal. 
This should give us the best chance to enlarge the matching if too many vertices would be left over, deriving a contradiction. 
A similar maximisation idea has also been used in~\cite{Treglown} when a degree sequence condition was given for a graph. 
The following Lemma will later guarantee the existence of an almost perfect matching in the reduced hypergraph.
\begin{lemma}[Matching]\label{matching_lemma}
	Let~$1/n\ll\alpha,\beta$.
	If $H=\left( [n],E\right)$ is a~$3$-graph,~$G_H$ a graph on vertex set~$[n]$ with maximum degree~$\Delta \left( G_H \right) \leq \beta n$ and~$H$ satisfies~$d(i,j) \geq \min\left( i,j, \frac{n}{2}\right) + \alpha n$, for all~$ij \in [n]^{(2)}$ with~$ij \notin E\left( G_H\right)$, then~$H$ has a matching~$M$ with~${v(M) \geq \left( 1-3\beta\right) n}$.
\end{lemma}

\begin{proof}[Proof of Lemma \ref{matching_lemma}]
	Without restriction let $\alpha \ll 1$ and $\beta < 1/3$ and let~$H,G_H$ be given as in the statement.
	For matchings~$M_1,M_2\subseteq H$ of maximal size, we write~$M_1\prec M_2$ if~$[n]\setminus V(M_1)\leq_{\text{lex}}[n]\setminus V(M_2)$, where~$\leq_{\text{lex}}$ is the usual lexicographic order on~$\cP([n])$, i.e.,~$A\leq B$ if~$\min A\triangle B\in A$.
	Now, let~$M\subseteq H$ be a matching of maximal size which is (subject to being of maximal size) maximal with respect to~$\prec$.
	Assuming the statement is false, gives an~${A\subseteq [n]\setminus V(M)}$ with~$\vert A \vert \geq 3\beta n$. Let us call a pair \textit{true} if it is not an edge in~$G_H$. 
	Since~$\Delta \left( G_H \right) \leq \beta n$, we can find~$2\beta n$ distinct vertices~$v_1, \dots, v_{\beta n}, w_1,\dots , w_{\beta n}\in A$ such that all the pairs~$v_iw_i$ are true. 
	Without restriction assume that~$v_i< w_i$.
	Notice that all the neighbours of each such pair lie inside~$V(M)$, otherwise adding the respective edge to~$M$ would lead to a larger matching.
	In the following, we will show two properties and afterwards deduce the statement from them.
	
	Firstly, we have that for each~$v_iw_i$, there are at least~$\frac{\alpha n}{3}$ edges in~$M$ in which~$v_iw_i$ has at least two neighbours: Let us first consider a pair~$v_iw_i$ with~$v_i \leq \frac{n}{2}$. 
	For any edge~$abc$ of the matching with~$a \in N\left( v_i,w_i\right)$, we have that~$\min\{b,c\}\leq v_i$ as otherwise~${E(M) \setminus \{abc\} \cup \{av_iw_i\}}$ would be the edge set of a matching~$M'$ with the same size as~$M$ but with~$M\prec M'$, contradicting our choice of~$M$.
	This means that in each edge of~$M$ which contains only one neighbour of~$v_iw_i$ there is one vertex~$\leq v_i$.
	Thus, (and since all those edges are disjoint), at most~$v_i$ neighbours of~$v_iw_i$ can lie in edges that contain no further neighbour of~$v_iw_i$.
	Hence, recalling~$d\left( v_i,w_i\right) \geq v_i + \alpha n$, at least~$\frac{\alpha n}{3}$ edges in~$M$ contain at least two neighbours of~$v_iw_i$. 
	
	For a pair~$v_iw_i$ with~$v_i \geq n/2$, there exist at least~$\frac{\alpha n}{3}$ edges in $M$ containing more than one neighbour of~$v_iw_i$ as well since~$d\left( v_i,w_i\right) \geq \frac{n}{2}+\alpha n$ but~${e(M) \leq n/3}$.
	
	Secondly, note that any edge of~$M$ that contains at least two neighbours of one true pair~$v_iw_i$ cannot contain a neighbour of any other true pair~$v_jw_j$: Assume for contradiction there were true pairs~$v_iw_i$ and~$v_jw_j$ together with an edge~$abc \in E(M)$ such that~$a\in N\left( v_i,w_i\right)$ and $\left\vert\{abc\} \cap N\left( v_j,w_j\right)\right\vert \geq 2$.
	Then~$b$ or~$c$, without restriction~$b$, is a neighbour of~$v_jw_j$ and~$E(M) \setminus \{abc\} \cup \{av_iw_i, bv_jw_j\}$ is the edge set of a matching in~$H$ contradicting the maximal size of~$M$.
	
	Summarised, for each of the~$\beta n$ true pairs~$v_iw_i$ in~$[n] \setminus V(M)$, we get a set of at least~$\frac{\alpha n}{3}$ edges in~$M$ that contain more than one neighbour of the respective pair and thus all those sets of edges are pairwise disjoint. 
	Therefore, we have~$\frac{\alpha n}{3}\times\beta n$ distinct edges in~$M$ which is a contradiction to~$1/n\ll\alpha,\beta$.
	So~$M$ was indeed a matching satisfying~$v(M)\geq\left( 1-3\beta\right) n$.
\end{proof}

We are now ready to prove Proposition \ref{long_path_lemma}. For that we will apply the Weak Regularity Lemma to~$H$ (actually to a slightly smaller subgraph), obtain a pair degree condition for the reduced hypergraph and hence find a matching in it by the previous Lemma. Lastly, we will ``unfold'' the edges of that matching to paths in~$H$ by Lemma~\ref{good_triplets_lemma} and connect these to a long path by the Connecting Lemma.

\begin{proof}[Proof of Proposition \ref{long_path_lemma}]
	Let~$\alpha ,\vartheta$ be given as in the Proposition and set~$\alpha ' = \alpha - \vartheta - \vartheta ^2$. 
	Next choose~$\xi , \delta , t_0$ such that we have~$1/t_0\ll\delta\ll\xi\ll\vartheta\ll\alpha '$.
	Applying the Weak Regularity Lemma~\ref{weak_regularity_lemma} to~$\delta$ and~$t_0$ gives us a~$T_0$ and by the hierarchy in the Proposition, we may assume~$1/n \ll 1/T_0$. 
	Now let~$H$,~$\mathcal{R}$, and~$P_A$ be given as in the statement. 
	Notice that~$H'=H\big[[n]\setminus\left(\mathcal{R}\cup V\left(P_A\right)\right)\big]$ after a renaming of the vertices can be seen as a~$3$-graph~$H'=\left([n'],E'\right)$ with~$n'\geq\left(1-\vartheta ^2 - \vartheta\right) n$ and satisfying the usual degree condition: $d(i,j) \geq \min \left( i,j,\frac{n'}{2}\right) + \alpha ' n'$ for all $ij \in [n']^{(2)}$.
	
	For~$H'$, the statement of the Weak Regularity Lemma provides an integer~$t\in[t_0,T_0]$ and a partition~$V = V_0 \dot\cup V_1 \dot\cup V_2 \dot\cup \dots \dot\cup V_t$ satisfying all three points of Lemma \ref{weak_regularity_lemma}.
	Setting~${m=\left\vert V_1 \right\vert = \dots = \left\vert V_t \right \vert}$, we have that~$\frac{n'}{t}\geq m \geq \frac{1-\delta}{t}n'$ and recall that~${\left\vert V_0 \right\vert \leq \delta n'}$.
	Note that for~$v_i \in V_i$, we have $v_i \geq i\cdot m -\frac{n'}{t_0}$.
	Summarised, we have the following hierarchy: 
	\begin{align}\label{hierarchy}
		\frac{1}{n'}\ll\frac{1}{T_0},\frac{1}{t},\frac{1}{t_0}\ll\delta\ll\xi\ll\vartheta\ll\alpha'\ll 1
	\end{align}
	Let us write $e^{\times}\left(V^{ijk}\right) =  \left\vert\left\{e\in E': \left\vert e\cap V_i\right\vert = \left\vert e\cap V_j\right\vert=\left\vert e\cap V_k\right\vert=1\right\} \right\vert$ for the number of crossing edges in $V^{ijk}$ and we call a triplet~$V^{ijk}$ dense, if~$e^{\times}\left( V^{ijk}\right) \geq \frac{\alpha 'm^3}{2}$.
	
	Now we will show that we can almost ``transfer'' the pair degree condition to a reduced hypergraph.
	We will do this in two steps: First, we show that every pair $V_iV_j$ belongs to many dense triplets $V^{ijk}$, and second, we show that we can almost keep that up when restricting ourselves to quasirandom triplets.
	\begin{claim}
		For every $ij\in [t]^{(2)}$, there are at least~$\min \left(i,j,\frac{t}{2}\right) + \frac{\alpha 't}{3}$ many $k\in [t]-\{i,j\}$ such that $V^{ijk}$ is a dense triplet.
	\end{claim}
	\begin{proof}
		Suppose there is a pair~$V_iV_j$,~$ij\in [t]^{(2)}$, belonging to less than~$\min \left(i,j,\frac{t}{2}\right) + \frac{\alpha 't}{3}$ dense triplets~$V^{ijk}$.
		Let~$S$ be the set of hyperedges in~$H'$ that contain one vertex in~$V_i$, one in~$V_j$ and a third vertex outside of~$V_i \dot\cup V_j$.
		By invoking the pair degree condition of~$H'$ and with the hierarchy (\ref{hierarchy}), we get that
		\begin{align*}
			\left\vert S \right\vert \geq & m^2 \left[ \min\left(i\cdot m -\frac{n'}{t_0},j\cdot m -\frac{n'}{t_0},\frac{n'}{2}\right) +\alpha 'n' -2m\right] \\
			> &\frac{n'^3}{t^2} \left(\min\left( \frac{i}{t},\frac{j}{t}, \frac{1}{2}\right) +\frac{6}{7} \alpha '\right)
		\end{align*}
		We will derive a contradiction by finding a smaller upper bound on~$\vert S\vert$.
		To this aim, we split~$S$ into two parts.
		By~$S_1$ let us denote the set of those edges in~$S$ that lie in a dense triplet~$V^{ijk}$, for some~$k\in [t]\setminus \{i,j\}$, (we say an edge~$e$ lies or is in~$V^{ijk}$ if we have~$\vert e\cap V_i\vert = \vert e\cap V_j\vert = \vert e\cap V_k\vert = 1$).
		Since in one triplet there are at most~$m^3$ edges and by assumption~$V_iV_j$ does not belong to many dense triplets, we get
		$$\left\vert S_1\right\vert < \left(\min \left(i,j,\frac{t}{2}\right) + \frac{\alpha 't}{3}\right) m^3 \leq \frac{n'^3}{t^2}\left(\min\left(\frac{i}{t},\frac{j}{t},\frac{1}{2}\right) +\frac{\alpha '}{3}\right)$$
		Let $S_2=S\setminus S_1$ be the set of edges in $S$ lying in triplets that are not dense. There are less than $\frac{\alpha '}{2}m^3$ crossing edges in each triplet that is not dense and $V_iV_j$ belongs to at most~$t$ triplets. Hence
		$$ \left\vert S_2\right\vert < \frac{\alpha '}{2}m^3 \times t \leq \frac{n'^3}{t^2} \frac{\alpha '}{2}.$$
		Summarised, we have 
		\begin{align*}
		\frac{n'^3}{t^2} \left(\min\left( \frac{i}{t},\frac{j}{t}, \frac{1}{2}\right) +\frac{6}{7} \alpha '\right)
		< \vert S\vert = \left\vert S_1\right\vert + \left\vert S_2\right\vert 
		< \frac{n'^3}{t^2}\left(\min\left(\frac{i}{t},\frac{j}{t},\frac{1}{2}\right) +\frac{5\alpha '}{6}\right)\,,
		\end{align*}
		which is a contradiction.
	\end{proof}
	From the Weak Regularity Lemma we also get that in total at most $\delta t^3$ triplets $V^{ijk}$ are not~$\delta$-quasirandom.
	
	Let us now complete the ``reduction'' of the hypergraph and notice that we can find an almost perfect matching in the reduced hypergraph.
	Denote by~$D$ the hypergraph on the vertex set~$[t]$ with~$ijk$ being an edge if and only if the triplet~$V^{ijk}$ is dense.
	Let, on the other hand,~$IR$ be the hypergraph on the vertex set~$[t]$ with~$ijk$ being an edge if and only if~$V^{ijk}$ is not weakly~$\delta$-quasirandom in~$H'$.
	In the following, we will remove a few vertices in such a way that~$D-IR$ induced on the remaining vertices satisfies our pair degree condition for almost all pairs.
	
	We call a pair~$ij \in [t]^2$ \textit{malicious pair} if it belongs to more than~$\sqrt{\delta }t$ edges of~$IR$. 
	Since~${e(IR) \leq \delta t^3}$, there are at most~$3\sqrt{\delta }t^2$ malicious pairs.
	Let~$B$ be the graph on vertex set~$[t]$ in which the edges are given by the malicious pairs.
	We call a vertex~$i$ \textit{malicious vertex} if~$d_B(i)>\delta ^{1/4}t$, i.e., if it belongs to many malicious pairs.
	The upper bound on the number of malicious pairs implies that there are at most~$6\delta ^{1/4}t$ malicious vertices.
	Now we remove these malicious vertices and set~${D' := D\big[[t]\setminus\{v\in [t]: v \text{ malicious}\}\big]}$ and~${B'=B\big[[t]\setminus\{v\in [t]: v \text{ malicious}\}\big]}$.
	
	The reduced hypergraph we looked for is now~$K=D'-IR$, in which edges encode dense,~$\delta$-quasirandom triplets. 
	In~$K$, every pair~$ij \in V(K)^{(2)}$ with~$ij \notin E\left[B'\right]$ satisfies 
	$$d_K(i,j) \geq \min\left(i,j,\frac{t}{2}\right) + \left(\frac{\alpha '}{3} - 6\delta ^{1/4} - \sqrt{\delta}\right) t\geq \min\left( i,j, \frac{t}{2}\right) +\frac{\alpha '}{4}t .$$
	Thus, we have that the graph~$G_K$ on vertex set~$V(K)$ with~$ij$ being an edge if and only if~$ij$ does not satisfy the degree condition~$d_K(i,j) \geq \min\left( i,j, \frac{v(K)}{2}\right) +\frac{\alpha '}{4} v(K)$ is a subgraph of~$B'$.
	Therefore, and since~$v(K)\geq(1-6\delta^{1/4})t$, we have $${\Delta \left( G_K\right) \leq \Delta \left( B'\right) \leq \delta ^{1/4} t}\leq 2\delta^{1/4}\vert V(K)\vert$$ and we can apply Lemma \ref{matching_lemma} to~$K$ with~$\frac{\alpha '}{4}$ in place of~$\alpha$ and~$2\delta ^{1/4}$ instead of~$\beta$ and obtain a matching~$M$ in~$K$ covering all but at most~$6\delta ^{1/4}t$ vertices of~$K$.
	
	Finally, notice that each triplet~$V^{ijk}$ with~$ijk$ being an edge in~$K$ is~$\left(\delta ,d_{ijk}\right)$-quasirandom with~$d_{ijk}\geq \frac{\alpha '}{2} - \delta \geq \frac{\alpha '}{3}$. 
	Hence, we may apply  Lemma \ref{good_triplets_lemma} (with~$\xi$ as in (\ref{hierarchy}),~$d_{ijk}\geq\frac{\alpha '}{3}$ in place of~$d$ and~$\delta$ as $\delta$) to each of the triplets $V^{ijk}$ that corresponds to an edge in~$M$.
	Doing so and recalling the definition of~$H'$, we notice that in~$H$ we can cover at least $$n-\left(\left( \delta + 6\delta ^{1/4} + 6\delta ^{1/4} +\xi \right) n' + \vert\mathcal{R}\vert + v\left( P_A\right)\right) \geq n-\left( 2\vartheta ^2 n + v\left( P_A\right)\right)$$ vertices with paths of length at least~$\frac{\frac{\alpha '}{3} \xi ^3 -\delta}{2}m-2$ that are all disjoint to~$\mathcal{R}$ and~$V\left(P_A\right)$.
	We can connect all those at most~${\frac{3t}{\frac{\alpha '}{3} \xi ^3 -\delta}}$ paths in~$H$ through~$\mathcal{R}$ to a path~$Q$ by Lemma~\ref{reservoir_preservation_lemma} since until we connect the last one we have still only used at most $$(L-2) \cdot \frac{3t}{\frac{\alpha '}{3} \xi ^3 -\delta} < \vartheta ^4 n$$ vertices from~$\mathcal{R}$ (recall the hierarchy (\ref{hierarchy})).
	In fact, we have that~$Q$ has at most a small intersection with~$\mathcal{R}$, that is,~$\left\vert V(Q) \cap \mathcal{R}\right\vert \leq \vartheta ^4 n$ and it covers many vertices, i.e.,~${v(Q) \geq \left( 1-2\vartheta ^2\right) n -v\left( P_A\right)}$. 
	Hence,~$Q$ is a path satisfying the claims in the statement.
\end{proof}

\section{Concluding Remarks}\label{concluding_remarks}
We would like to finish by pointing to some related problems. Firstly, as mentioned in the introduction, our result can be seen as a stepping stone towards a complete characterisation of those pair degree matrices that force a~$3$-graph to contain a Hamiltonian cycle.

Further, it seems possible to generalise our proof without too much effort for~$k$-uniform hypergraphs~$H=([n],E)$ with~$n$ large satisfying the~$(k-1)$-degree condition $$d_{k-1}(i_1,\dots ,i_{k-1})\geq\min\left(i_1,\dots,i_{k-1},\frac{n}{2}\right)+\alpha n \,,$$ where~$d_{k-1}(i_1,\dots ,i_{k-1})=\left\vert\right\{e\in E: \{i_1,\dots ,i_{k-1}\}\subseteq e\left\}\right\vert$.

Another very interesting problem is to get a similar result for the vertex degree, strengthening the result by Reiher, R\"odl, Ruci\'nski, Schacht, and Szemer\'edi in \cite{5/9}: Does every~$3$-graph $H=([n],E)$ with $d(i)\geq \min\left(\max\left(i,\gamma n\right),\frac{5}{9}n\right)+\alpha n$ for some~$\gamma < 5/9$ contain a Hamiltonian cycle if~$n$ is large? The proof of Theorem \ref{5/9_theorem} in \cite{5/9} depends on the existence of \textit{robust subgraphs} for every vertex, for which one needs the factor $5/9$.

Lastly, one could try to improve Theorem \ref{main_theorem} by weakening the pair degree condition to $d(i,j)\geq \min\left(i,j,\frac{n}{2}\right)$, i.e., without the additional~$\alpha n$ term, as R\"odl, Ruci\'nski, and Szemer\'edi did for the minimum pair degree condition in~\cite{Dirac_type_improved}.

\section{Acknowledgement}
This article is based on my master thesis from summer 2018 which was supervised by Christian~Reiher.
I would like to thank him for introducing me to the absorption method and to this problem.
Further, I would like to thank an anonymous referee for suggesting to try to improve the original result.

\end{document}